\documentclass{article}

\usepackage[latin1]{inputenc}
\usepackage{amsmath}
\usepackage{amsfonts}
\usepackage{latexsym}
\usepackage[psamsfonts]{amssymb}
\usepackage[psamsfonts]{eucal}
\usepackage[active]{srcltx} 
\usepackage{pstricks}
\usepackage{graphicx}
\usepackage{subfigure}

\def\R{{\mathbb R}}
\def\Rn{{{\R}^N}}

\def\Haus{{{\mathcal H}}}

\def\Hnmu{{\Haus^{N-1}}}
\def\vrho{\varrho}

\def\diver{\mathop{\rm div}\nolimits}

\newcommand\scal[2]{{\left\langle #1 ,#2\right\rangle}}

\newtheorem{theo}{Theorem}[section]
\newtheorem{prop}[theo]{Proposition}
\newtheorem{cor}[theo]{Corollary}
\newtheorem{lemma}[theo]{Lemma}
\newtheorem{conjecture}[theo]{Conjecture}
\newcommand{\qed}{\thinspace\null\nobreak\hfill\hbox{\vbox{\kern-.2pt\hrule
height.2pt depth.2pt\kern-.2pt\kern-.2pt \hbox to2.5mm{\kern-.2pt\vrule
width.4pt \kern-.2pt\raise2.5mm\vbox to.2pt{}\lower0pt\vtop
to.2pt{}\hfil\kern-.2pt \vrule
width.4pt\kern-.2pt}\kern-.2pt\kern-.2pt\hrule height.2pt depth.2pt
\kern-.2pt}}\par\medbreak}

\newtheorem{ex1}[theo]{Example}
\newenvironment{ex}{\begin{ex1} \rm }{\end{ex1} }
\newtheorem{rmk1}[theo]{Remark}
\newenvironment{rmk}{\begin{rmk1} \rm }{\end{rmk1} }

\newcommand{\la}{\lambda}
\newcommand{\al}{\alpha}
\newcommand{\be}{\beta}
\newcommand{\g}{\gamma}
\newcommand{\ep}{\epsilon}
\newcommand{\ci}{{\cal C}}
\newcommand{\de}{\delta}
\newenvironment{proof}{\noindent{\sc Proof.}}{\qed}
\newcommand{\f}{{\phi}}
\newcommand{\rr}{\mathbb{R}}
\newcommand{\p}{\psi}

\newtheorem{que1}{Question }
\newenvironment{que}{\begin{que1} \rm }{\end{que1} }

\newcommand{\arc}{{\rm arc}}

\title{Some isoperimetric problems in planes with density}
\author{
Antonio Ca\~{n}ete\footnote{
Departamento de Matem\'atica Aplicada I, Universidad de Sevilla,
Campus de Reina Mercedes, E--24071, Sevilla, Espa\~na; e-mail: antonioc@us.es},
Michele Miranda Jr\footnote{
Dipartimento di Matematica, Universit\`a degli Studi di
  Ferrara, via Machiavelli 35, 44100 Ferrara, Italy; e--mail:
  michele.miranda@unife.it},
and Davide Vittone\footnote{
Dipartimento di Matematica Pura ed Applicata, Universit\`a
  degli Studi di Padova, via Trieste 63, 35100 Padova (PD), Italy;
  e--mail: vittone@math.unipd.it}}

\begin{document}

\maketitle

\begin{abstract}
We study the isoperimetric problem in Euclidean space endowed with a density.
We first consider piecewise constant densities and examine particular
cases related to the characteristic functions of half-planes, strips and balls.
We also consider continuous modification of Gauss density in
$\R^2$. Finally, we give a list of related open questions.
\end{abstract}

\section{Introduction}
The classical isoperimetric problem in Euclidean space $\R^N$ looks for 
``least perimeter'' sets 
among those ones with prescribed volume. A first solution of this problem was provided 
in 1935 by Lusternik, by proving, using symmetrisation techniques, that balls minimise 
Minkowski content 
(see also Burago-Zalgaller \cite[Chapters 8 and 10]{BurZal88Geo}); the complete solution 
was given by De Giorgi \cite{DeG58Sul},
with the proof that balls are minimisers in the class of finite perimeter sets. 
The isoperimetric problem has been extended to Riemannian manifolds and, more recently, 
to the so-called {\em manifolds with density}. 

A manifold with density (see also \cite[Chapter 18]{Mor08Geo}) is a manifold $M$ endowed 
with a positive function $f:M\to\R^+$ used to 
weight both perimeter and volume. More precisely, given a regular set $E\subset M$, 
we define the (weighted) 
volume and perimeter of $E$ to be
$$
|E|_f:=\int_E f\,dv,\qquad P_f(E):=\int_{\partial E} f\,da,
$$
where $dv$ and $da$ are volume and surface elements on $M$ (see also Section \ref{preliminaries} 
for precise definitions). Note that such a density is not equivalent to scaling the metric of $M$
conformally by some factor $\la$, since in that case 
perimeter and volume would scale by different powers of $\la$.

One of the first and most important examples of manifolds with density, with applications 
to probability and statistics, is Gauss space $\R^N$ with density $f(x)=e^{-|x|^2}$. 
Isoperimetric sets are half-spaces (\cite{Bor75The,SudTsi78Ext, CarKer01OnT}; see also 
\cite[Chapter 18]{Mor08Geo} or \cite[Section 3]{Ros05The}). More recently other authors 
(e.g. \cite{Ros05The, BayCanMorRos07OnT, CJQW})
have considered different examples of manifolds with density and generalised classical techniques: 
Steiner  and Schwarz symmetrisations \cite{Ros05The}, functional 
versions of the isoperimetric inequality \cite{Bob97AnI}, first and second variation formulae, mean 
curvature, Ricci curvature, and so on. We refer the reader also to \cite{Mor05Man, Mor08Geo} 
and references therein.

In this paper we will consider two kinds of densities on the plane $\R^2$: continuous modifications of
the Gaussian density and piecewise constant densities, apparently the first examples of discontinuous
ones. These questions arose in the lectures on ``Manifolds with density'' by F.~Morgan 
during the meeting ``GMTLAP-Geometric Measure Theory and Least Area Problems'' 
in Modena (Italy), 15-17th February 2007, organized 
by G.P.~Leonardi, R.~Monti, F.~Serra~Cassano, R.~Serapioni and I.~Tamanini.

The plan of the paper is the following. Section \ref{preliminaries} contains introductory material 
on manifolds
with density, perimeter, existence of isoperimetric sets, and first variation of the perimeter. 
We stress Proposition \ref{teofresnel}, which provides a Snell
refraction law for isoperimetric boundaries for discontinuous, piecewise constant densities.
We comment on minimal surfaces in manifolds with density and introduce calibrations, 
which provide sufficient conditions for minimality.

In Section \ref{SecPCDens} we investigate the isoperimetric problem in some cases of piecewise 
constant densities. 
The first one is the half-space density, i.e., the density on $\R^N$ 
taking values 1 on the half-space $\{x_N\le0\}$ and $\la>1$ on $\{x_N>0\}$, 
where $\la$ is a function depending only on the last coordinate $x_N$; 
in this case, the isoperimetric sets are provided by round balls 
in the half-space $\{x_N\le 0\}$ (Theorem~\ref{Serapioni}). 
The second one is the ``strip'' density on $\R^N$ taking values 1 if $|x_N|\leq 1$ and a
constant $\la>1$ otherwise; 
we show that isoperimetric sets exist and are connected for any given volume in the planar case, 
with some additional properties in general dimension. 
A further analysis leads us to Conjecture~\ref{conj:strip}, 
where we list the possible solutions in the planar setting. 
This conjecture is almost completely proved in Theorem~\ref{teostrip}.
The last density we consider is the 
``ball'' density on $\R^2$ with values $1$ outside the ball $B(0,1)$ and constant $\la\neq1$ inside: here isoperimetric 
sets exist for any given volume, are connected, and can assume several profiles, both in
the case $\la>1$ (see Theorem~\ref{ballmagg})
and the case $\la<1$ (see Theorem \ref{conj:ballmino1}). 

Section \ref{SecContDens} concerns with modifications of Gauss density on $\R^2$. A question risen during the GMTLAP meeting concerned the density $ye^{-(x^2+y^2)/2}$ on the half-plane $H=\{(x,y)\in \R^2 : y>0\}$. We recently discovered that a paper of Brock et al. \cite{BroChiMer08ACl} gives a complete answer to this problem, showing that sets bounded by vertical lines are isoperimetric. We point out that, from that result, the same holds true also for the whole plane. During the GMTLAP meeting also the isoperimetric problem in the plane with density $e^{-x^2-y^4}$ was risen: we conjecture that here the isoperimetric boundary is provided by vertical or horizontal lines. For certain volume bounds (those close to half of
the total volume) horizontal lines are better than vertical ones, while for other bounds vertical lines beat horizontal ones (see Proposition~\ref{propconj1}).


The last Section \ref{SecOpenQue} collects several open questions: some from this paper, others from the GMTLAP meeting.

{\small {\bf \em Acknowledgments.} A.~C. is partially supported by the MCyT research project MTM2007-61919;
D.V. was partially supported by University of Padova, GNAMPA of MIUR and GALA project - Geometric Analysis in Lie groups and Applications, supported
by the European Commission. 
The authors want to thank specially Frank Morgan, who inspired this work during the GMTLAP meeting, 
read kindly first manuscripts and gave important ideas and improvements during the preparation of this paper.

We further want, among the others, to acknowledge:
Giovanni Alberti, Susanna Ansaloni, Nicola Arcozzi, Francesco Bigolin, Alessio Brancolini, Luca Corbo Esposito, 
Thomas Meinguet, Laurent Moonens, Fabio Paronetto, C\'esar Rosales, Raul Serapioni, Francesco Serra Cassano, Jeremy Tyson.}

\section{The isoperimetric problem in a space with density}\label{preliminaries}

Our framework is an open subset $\Omega\subset \Rn$ (we deal primarily with the cases
of the whole space and the upper half-plane) endowed with a lower--semicontinuous, positive 
density function $f:\Omega\to \R^+$; due to the positivity of $f$, we shall sometimes write
$f=e^\psi$. For a given measurable set $E\subset \Rn$, we define (according to \cite{Mor05Man}) its $f$--volume by
$$
|E|_f =\int_{E\cap \Omega} f(x)dx
$$
and its $f$--perimeter measure by
$$
P_f(E,\Omega)= 
\liminf_{h\to +\infty}
\int_{\partial E_h \cap \Omega} fd\Hnmu
$$
where the $\liminf$ is taken among all smooth sets $E_h$ converging to $E$ in the
$L^1(\Omega)$--topology.
If no confusion may arise, we will just say volume and perimeter.
We say that $E$ has finite perimeter in $\Omega$ if
$P_f(E,\Omega)<+\infty$; whenever the set $\Omega$ is clear from the context,
we shall simply write $P_f(E)$. The lower--semicontinuity of $f$ ensures the equality 
$$
P_f(E)=\int_{\partial E\cap \Omega} f(x)d\Hnmu(x)
$$
for smooth sets. 

We also notice that, by restricting to the open sets 
$$
\Omega_k=\{x\in \Omega: 1/k<f(x)<k\},
$$
finite perimeter sets and finite Euclidean perimeter sets
in $\Omega_k$ coincide for all $k>0$. By well--known results on finite
perimeter sets (see \cite{DeG54SuU, DeG55Nuo, DeG06Sel, Fed69Geo}), 
there exists a subset ${\cal F}E$ of the
topological boundary $\partial E$ which is rectifiable; that is, up to
$\Hnmu$--negligible sets, is contained in a countable union of Lipschitz
graphs. We point out that for a $(N-1)$--rectifiable set $\Sigma$ the unit normal vector
$\nu_\Sigma$ is defined $\Hnmu$--almost everywhere.
The set ${\cal F}E$ is called {\em reduced boundary} of $E$; for sake of
simplicity, 
we will sometimes write $\Sigma$ instead of ${\cal F}E$ and, if not
differently stated, $\nu_\Sigma$ will denote the inward unit normal
vector. 
Mostly important, this part of the boundary is the
essential part for the perimeter measure; that is, the following representation
formula holds:
$$
P_f(E,\Omega_k)=\int _{{\cal F}E\cap \Omega_k} f(x)d\Hnmu(x). 
$$
Using the fact that $\Omega=\cup \Omega_k$, we also obtain that
$$
P_f(E)=\int_{{\cal F}E\cap \Omega} f(x)d\Hnmu(x).
$$

We shall use next result.

\begin{lemma}
If the density $f$ satisfies $f\geq \eta>0$ on $\Omega$, then $P_f(E,\Omega)\geq \eta P(E,\Omega)$; 
moreover, in the case $N=2$, if $E$ is smooth and connected, we have 
\begin{equation}\label{estdiamper}
{\rm diam}_\Omega (E)\leq \frac{1}{2}P(E,\Omega) \leq \frac{1}{2\eta} P_f(E,\Omega). 
\end{equation}
\end{lemma}


Since we allow discontinuities in $f$, we introduce the 
following notation for $f\in L^1_{\rm loc}(\Omega)$. For fixed $\nu\in {\mathbb S}^{N-1}$, let
$$
B^+_\nu (x,\vrho)=\{x\in B(x,\vrho): \scal{x}{\nu} \geq 0\}, \quad
B^-_\nu (x,\vrho)=\{x\in B(x,\vrho): \scal{x}{\nu} \leq 0\},
$$
where $B(x,\vrho)$ is the open ball of center $x$ with radius $\vrho$ (balls, unless stated otherwise, 
are always assumed to be open). Moreover, we define
the limits of $f$ in the positive and negative direction of the vector $\nu$ by
\begin{equation}\label{AppLimit}
f^{\pm}_\nu(x)=\lim_{\vrho\to 0} \frac{1}{|B^{\pm}_\nu(x,\vrho)|} \int_{B^{\pm}_\nu(x,\vrho)} f(y)dy.
\end{equation}
The quantity $(f^+_\nu-f^-_\nu)\nu$ does not change if we replace $\nu$ with $-\nu$, so we will assume, if not
stated otherwise, that $\nu$ is chosen in such a way that $f^-_\nu\leq f^+_\nu$.
If $X$ is any nonzero vector, we denote
by $f^{\pm}_X$ the previous quantities in the direction $\nu=\frac{X}{|X|}$.


We collect in the following Proposition a little list of known results in the isoperimetric problem 
we shall use and generalise. A very nice description of these properties is contained in
\cite{BayCanMorRos07OnT}. We recall that a variation is given by a one--parameter family of diffeomorphisms
$\Phi_t:\Rn\to \Rn$; for such a given family, we define 
$$
V(t)=|\Phi_t(E)|_f, \qquad
P(t)=P_f(\Phi_t(E)).
$$

\begin{prop}
In a region $\Omega$ with density $f$, consider the isoperimetric problem of finding a closed
``isoperimetric set'' $E$ of least perimeter for prescribed volume $v$.
\begin{enumerate}
\item
If $|\Omega|_f<+\infty$, an isoperimetric set always exists.
\item
If $f$ is Lipschitz continuous, then $E$ is locally 
$\ci^{1,1}$ except for a singular set of Hausdorff codimension $8$ in $\Omega$.
\end{enumerate}
\end{prop}

\begin{proof}
We sketch here the proofs; we fix $v\in (0,|\Omega|_f)$ and a minimising sequence of sets $E_h$ with $|E_h|_f=v$. 
Condition $|\Omega|_f<+\infty$, by setting $d\mu=fdx$ implies that the functions $\chi_{E_h}$ are
equi--integrable in $L^1(\Omega,\mu)$, so they are relatively compact. Up to subsequences, $\chi_{E_h}$ converges 
to a function $g\in L^1(\Omega,\mu)$; by the a.e. convergence, $g$ has to be a characteristic function $\chi_E$.
The fact that $|E|_f=v$ follows by continuity of the measure and
$$
P_f(E)\leq \liminf_{h\to +\infty}P_f(E_h)
$$
follows by lower--semicontinuity of the perimeter measure, so $E$ is an isoperimetric set. 

The proof of regularity can be found in Morgan \cite[Proposition 3.5, Corollary 3.8 and Remark 3.10]{Mor03Reg}.
\end{proof}

\begin{rmk}\label{rmkMainProperties}
We collect here some necessary conditions for a set $E$ to be isoperimetric, under the assumption that $f$ is $C^1$
(see Rosales et al. \cite{BayCanMorRos07OnT}):
\begin{enumerate}
\item
$\Sigma$ is stationary under volume preserving variations, i.e. 
$P'(0)=0$ whenever $V'(0)=0$;
\item
there exists a constant $H_0$ such that for any variation
\begin{equation}\label{varprima}
(P-H_0V)'(0)=0;
\end{equation}
\item
the generalised mean curvature 
\begin{equation}\label{first_variation}
H_\psi=(N-1)H_\Sigma-\scal{\nabla \psi}{\nu_\Sigma},
\end{equation}
with $H_\Sigma$ the Euclidean mean curvature, is constant. 
\end{enumerate}
\end{rmk}

In Section \ref{SecPiecConstGen} we shall see how to generalise these conditions
to the case of piecewise regular density.

\subsection{Minimal surfaces and calibrations}
We extend to manifolds with density the classical method of
calibrations (see e.g. Definition 4.1 and Theorem 4.2 of Harvey-Lawson 
\cite{HarLaw82Cal} or Federer \cite{Fed69Geo}), giving sufficient conditions for a surface 
to be perimeter minimising. 

We say that a set $E\subset\R^N$ of finite perimeter is {\em perimeter
minimising} (without volume constraint) in an open set $\Omega\subset\R^N$ if
$$
P_f(E,\Omega)\leq P_f(F,\Omega)
$$
for any set $F$ such that $E\Delta F:=(E\setminus F)\cup(F\setminus E)\Subset\Omega.$ 

\begin{theo}\label{calibrations}
Suppose $\Omega\subset \Rn$ is an open set and $E\subset\R^N$ has finite perimeter in $\Omega$.
Moreover, suppose that there exists a
$\ci^1$ function $g:\Omega\to\R^N$ such that 
\begin{itemize}
\item $|g(x)|\leq 1$ for any $x\in\Omega$;
\item $g\equiv\nu_E\ \Hnmu$-a.e. on ${\cal F}E$; 
\item $\diver (fg)=0$ in the distributional sense.
\end{itemize}
Then $E$ is perimeter minimising in $\Omega$.
\end{theo}
\begin{proof}
Let $F$ be another open set of finite perimeter in $\Omega$ with
$E\Delta F\Subset\Omega'\Subset\Omega$, where without loss of
generality we assume $\Omega'$ to have smooth boundary with finite
surface measure. Since $\diver (fg)=0$, by the divergence theorem 
\begin{eqnarray*}
P_f(E,\Omega) &=& P_f(E,\Omega\setminus\Omega')+\int_{{\cal F}
  E\cap\Omega'} fd\Hnmu \\ 
&=& P_f(E,\Omega\setminus\Omega')+\int_{{\cal F} E\cap\Omega'} \scal{fg}{\nu_E} d\Hnmu\\ 
&=& P_f(E,\Omega\setminus\Omega')- \int_{\partial \Omega'\cap E}
\scal{fg}{\nu_{\Omega'}} d\Hnmu\\ 
&=& P_f(F,\Omega\setminus\Omega')- \int_{\partial \Omega'\cap F}\scal{fg}{\nu_{\Omega'}}d\Hnmu\\ 
&=& P_f(F,\Omega\setminus\Omega')+\int_{{\cal F} F\cap\Omega'}
\scal{fg}{\nu_F}d\Hnmu\\ 
&\leq& P_f(F,\Omega\setminus\Omega')+P_f(F,\Omega')\ =\ P_f(F,\Omega).
\end{eqnarray*}
\end{proof}

\begin{ex}
Consider $\R^2$ with density $f(x,y)=(x^2+y^2)^{1/2}$; then the curve
$(x,\sqrt{1+x^2})$ is area minimising, in the sense that its
epigraph is perimeter--minimising. In fact, a calibration is given by 
$$
g(x,y):=\left(-\frac{x}{\sqrt{x^2+y^2}},\frac{y}{\sqrt{x^2+y^2}}\right)
$$
which extends the normal vector to the curve and satisfies the
hypotheses of Theorem~\ref{calibrations}. Notice that the singularity
of $g$ in 0 does not affect the validity of our argument. 
\end{ex}

\begin{rmk}
Suppose $S$ is a hypersurface in $\R^N$ (with a $\ci^1$ regular density $f$) which
coincides with the graph of a $\ci^1$ function $\f:\R^{N-1}\to\R$,
i.e. 
$$S=\{(x,x_N):x\in\R^{N-1},x_N=\f(x)\}.$$
It is immediate to see that the surface measure of $S$ is given by
$$
\sigma_f(S)=\int_{\R^{N-1}}\sqrt{1+|\nabla\f(x)|^2}\,f(x,\f(x)) dx.
$$ 
For area--minimising hypersurfaces the generalised
mean curvature introduced in (\ref{first_variation}) vanishes;
this means that $\f$ satisfies the minimal surface equation 
\begin{equation}\label{MSE}
\diver\left(\frac{f(x,\f(x)) \nabla\f(x)}{\sqrt{1+|\nabla\f(x)|^2}}\right) 
+ \sqrt{1+|\nabla\f(x)|^2} \frac{\partial f}{\partial x_N}(x,\f(x))=0
\quad\text{in }\R^{N-1}\,. 
\end{equation}
\end{rmk}

\begin{rmk}
We observe in passing that, in Gauss space $\R^N$ with density
$e^{-|x|^2}$, half--spaces through the origin have boundary (seen as a
graph) satisfying~\eqref{MSE} (they have vanishing mean curvature) but
are not perimeter minimising in $\R^N$, although they are isoperimetric for
volume $1/2$ \cite{Bor75The,SudTsi78Ext}. It can be proved in fact that there are no
complete perimeter--minimising surfaces in Gauss space. Suppose by contradiction
that there exists a measurable set $E$ which is perimeter--minimising
in $\R^N$. Choose an open bounded
set $\mathcal U$ such that $P_f(E,\mathcal U)>0$. Let us consider a
sufficiently large $R>0$ such that $\mathcal U\Subset B_R=B(0,R)$ and 
$$
P_f({B_R},\R^N)=N\omega_NR^{N-1}e^{-R^2}<P_f(E,\mathcal U),
$$
and set $F:=E\setminus B_R$. It is easily seen that
\begin{align*}
P_f(F,B_{R+1})&\leq P_f(F,B_{R+1}\setminus\overline{B_R})+P_f(B_R,\R^N)\\
&< P_f(F,B_{R+1}\setminus\overline{B_R}) + P_f(E,\mathcal U) \leq P_f(E,B_{R+1}),
\end{align*}
contradicting the minimality of $E$.
\end{rmk}

\begin{theo}\label{calgraphs}
Suppose that $f$ is a regular density in $\R^N=\R^{N-1}_x\times\R_{x_N}$ which
does not depend on the last coordinate $x_N$, and suppose that
$\f: \omega\subset \R^{N-1}\to\R$ satisfies the minimal surface
equation~\eqref{MSE}. Then the graph of $\f$ is perimeter--minimising in
$\omega\times\R$. 
\end{theo}
\begin{proof}
With our hypotheses on $f$ the minimal surface
equation~\eqref{MSE} becomes 
\begin{equation}\label{MSE*}
\diver_{\R^{N-1}}\left(\frac{f
  \nabla\f}{\sqrt{1+|\nabla\f|^2}}\right)=0 \quad\text{in }\omega\,, 
\end{equation}
where of course we interpret $f$ as a function of $N-1$ coordinates
only. Then the normal to the graph at a point $(x,\f(x))$ is given by 
$$\left( \frac{\nabla\f(x)}{\sqrt{1+|\nabla\f(x)|^2}}\,,\,
-\frac{1}{\sqrt{1+|\nabla\f(x)|^2}} \right),$$ 
and if we set
$$g(x,x_N):= \left( \frac{\nabla\f(x)}{\sqrt{1+|\nabla\f(x)|^2}}\,,\,
-\frac{1}{\sqrt{1+|\nabla\f(x)|^2}}
\right),\qquad(x,x_N)\in\omega\times\R,$$ 
by~\eqref{MSE*} we obtain
$$\diver_{\R^N}\, (fg)=\diver_{\R^{N-1}}\left(\frac{f
  \nabla\f}{\sqrt{1+|\nabla\f|^2}}\right)=0\,.$$ 
Therefore the hypotheses of Theorem~\ref{calibrations} are satisfied,
  and the graph of $\f$ is perimeter--minimising.
\end{proof}

\begin{cor}\label{corhyperplanes}
Suppose $\R^N$ is endowed with a density independent on the last
coordinate $x_N$; then all affine hyperplanes $\{x_N=const\}$ are
perimeter--mi\-ni\-mising surfaces.
\end{cor}

\begin{ex}
Let us consider the plane $\R^2$ endowed with the density $f(x,y)=e^x$ 
introduced in \cite{CJQW}, where it is proved (see Corollary 4.7 therein) that there exist no isoperimetric regions, essentially because constant curvature curves have infinite perimeter. As observed in
Corollary~\ref{corhyperplanes}, hyperplanes $\{y=const\}$ are perimeter--minimising
surfaces. Using Theorem~\ref{calgraphs}, one could check that for any
$a,b\in\R$ the graphs of functions 
$$
\f(x)=
a\pm\int_0^x
\sqrt{\frac{1}{e^{2t}(1+\frac{1}{b^2})-1}}dt,
$$ 
defined on proper half-lines $]c,+\infty[$, are perimeter--minimising
surfaces in $]c,+\infty[\times\R$. 
\end{ex}

\subsection{Piecewise regular densities}
\label{SecPiecConstGen}

In the case of a piecewise regular density, 
variational formulae have to be slightly modified taking into
account the jump set of the density $f$. By piecewise regular
density we mean a function
$$
f(x)= 
\left\{\begin{array}{ll}
f_i(x) & x\in \Omega_i \\
\mbox{ } & \\
\inf \{f_i(x)\} & x\in \partial \Omega_i  
\end{array}
\right.
$$
with $\Omega_1,\ldots,\Omega_k$ disjoint open Lipschitz domains, $f_i\in \ci^1(\overline{\Omega}_i)$ ,
and $f>0$ on $\overline\Omega_i$: we denote by $\Gamma$ the set
$$
\Gamma=\bigcup _{i=1}^k \partial \Omega_i.
$$
Using the notation introduced in (\ref{AppLimit}), for $\Hnmu$-almost every $x\in \Gamma$ we have
$f^\pm_\nu(x)=f_i(x)$ for some $i$, where $\nu=\nu_\Gamma$ is
a unit normal vector to $\Gamma$; as usual, $\nu_\Gamma$ has the 
property that $f^-_{\nu_\Gamma}\leq f^+_{\nu_\Gamma}$. 
For a given 
set $E\subset \Rn$ with finite perimeter,
the functions $\chi_{E,\nu_\Gamma}^\pm$ are the traces of $E$
on the two sides of $\Gamma$. These traces are well defined since the trace
operator $T:BV(\Omega)\to L^1(\partial\Omega)$ is continuous whenever
$\partial \Omega$ is Lipschitz (see for instance \cite[Theorem 3.88]{AmbFusPal00}).

We can then state the following proposition.

\begin{prop}[First variation of volume and perimeter]\label{PropVarVolPerPC}
Let $f$ be a piecewise regular den\-sity and let $E$ be a set of finite perimeter, $\Sigma$ its reduced
boundary. Let us assume that $\{\Phi_t\}_{t\geq 0}:\Rn\to \Rn$ is a smooth one-parameter variation with $\Phi_0=Id$. Set 
$X:=\frac{d}{dt}\Phi{_t}{_{|t=0}}$ and $u:=\scal{X}{\nu_\Sigma}$,
 $\nu_\Sigma$ the inward unit normal vector to ${\cal F}E$; then the following first variation formula 
for the volume holds
\begin{eqnarray}
V^\prime(0) 
&=&\int_E (\scal{\nabla f}{X} + f\diver X) dx \nonumber\\
&=&-\int_{\Sigma\setminus \Gamma} fu d\Hnmu - \int _{\Gamma} 
(f^+_{\nu_\Gamma}\chi^+_{E,\nu_\Gamma}-f^-_{\nu_\Gamma}\chi^-_{E,\nu_\Gamma}) 
\scal{X}{\nu_\Gamma} d\Hnmu.\label{firstvarVolPC}
\end{eqnarray}
Moreover, the function $P(t)=P_f(\Phi_t(E))$
is differentiable at $t=0$ if for $\Hnmu$-a.e. $x\in \Gamma\cap \Sigma$ there exists
some $t_x>0$ such that $\Phi_t(x)\in \Omega_i$ for $0<t<t_x$, where $i$ is such that
$f^-_{\nu_\Gamma}(x)=f_i(x)$; this is the case if, for instance,
\begin{equation}\label{posScal}
\scal{X}{\nu_\Gamma}<0 
\end{equation}
wherever $\nu_\Gamma$ is defined. 
In this case
\begin{equation}\label{firstVarPer}
P^\prime(0)=\int_\Sigma f^-_{\nu_\Sigma} \diver_\Sigma X d\Hnmu + \int_{\Sigma} 
\scal{\nabla f^-_{\nu_\Sigma}}{X} d\Hnmu.
\end{equation}
\end{prop}

For the proof of Proposition \ref{PropVarVolPerPC}, we shall make use of the 
following result \cite[Lemma 10.1]{Giu84Min}.

\begin{lemma}\label{lemmaGiusti}
Let $f$ be any measurable lower--semicontinuous density, and let
$F:\Rn\to \Rn$ be a diffeormorphism.
If $E$ and $F(E)$ have finite perimeter, then the following relation holds
\begin{equation}\label{ChangeofVar}
P_f(F(E)) =\int_{{\cal F}E} f(F(x))|H_F(x)\nu_E(x)| d\Hnmu(x),
\end{equation}
where we have defined
$$
H_F(x)=|{\rm det}DF| DF^{-1}(F(x)).
$$
\end{lemma}

We can then prove Proposition \ref{PropVarVolPerPC}.

\begin{proof}
The proof of this fact is essentially the same contained in \cite{BayCanMorRos07OnT}; in fact, we have that
$$
V(t)=\int_{\Phi_t(E)} f(y)dy =\int_E f(\Phi_t(x))|{\rm det}D\Phi_t(x)|dx.
$$ 
Then, using the almost everywhere differentiability of $f$,
\begin{align*}
V^\prime(0)=&\int_E \left(
\scal{\nabla f(x)}{X(x)} +f(x)\frac{d}{dt}|{\rm det}D\Phi_t(x)|_{t=0}\right)dx \\
=&
\sum_{i=1}^k \int_{E\cap \Omega_i} \left( 
\scal{\nabla f_i(x)}{X(x)}+f_i(x)\diver X(x)\right)dx. 
\end{align*}
Notice that $E\cap \Omega_i$ has finite perimeter and 
${\cal F}(E\cap \Omega_i)$ is made by two pieces: ${\cal F}E\cap \Omega_i$ and the
part of $\partial \Omega_i$ where $E$ has trace $1$, that is $\chi_{E,\nu_{\Omega_i}}^+=1$ (here $\nu_{\Omega_i}$ 
is the inward unit normal to $\Omega_i$). Therefore, the divergence theorem gives
\begin{align*}
\int_{E\cap \Omega_i} \scal{\nabla f_i(x)}{X(x)}dx
=&\int_{{\cal F}E\cap \Omega_i} f_i\scal{X}{\nu_E} \,d\Hnmu\\
&+\int_{\partial \Omega_i} f_i
\chi_{E,\nu_{\Omega_i}}^+ \scal{X}{\nu_{\Omega_i}}\,d\Hnmu 
-\int_{E\cap \Omega_i} f_i\diver X dx.
\end{align*}
Summing up on $i=1,\ldots,k$, formula (\ref{firstvarVolPC}) follows.

In order to prove (\ref{firstVarPer}), we use (\ref{ChangeofVar}) with $F(x)=\Phi_t(x)$
and $H_t=H_{\Phi_t}$; we then have
$$
P_f(\Phi_t(E)) =\int_{{\cal F}E} f(\Phi_t(x)) |H_t(x)\nu_E(x)|d\Hnmu
$$
and then
\begin{align*}
\frac{d}{dt}P_f(\Phi_t(E)) =&\int_{{\cal F}E} \frac{d}{dt}
\left(f(\Phi_t(x))\right) |H_t(x)\nu_E(x)|d\Hnmu \\
&+
\int_{{\cal F}E} f(\Phi_t(x)) \frac{d}{dt}|H_t(x)\nu_E(x)|d\Hnmu.
\end{align*}
For the second term on the right hand side, we have as in the standard case that
$$
\frac{d}{dt}|H_t(x)\nu_E(x)|_{|t=0} =\diver_\Sigma X(x)
$$
with $\diver_\Sigma$ the tangential divergence; since $|H_0(x)\nu_E(x)|=1$, we have only to compute
$$
\frac{d}{dt}f(\Phi_t(x))_{|t=0}.
$$
We analyse separately the two cases: $x\in \Omega_i$ for some $i$, and 
$x\in \Gamma$. 
For $x\in \Omega_i$, we have that $\Phi_t(x)\in \Omega_i$ for
small $t$, and then
$$
\frac{d}{dt}f(\Phi_t(x))=\scal{\nabla f_i(\Phi_t(x))}{\tfrac{d}{dt}\Phi_t(x)},
$$ 
whence
$$
\frac{d}{dt}f(\Phi_t(x))_{|t=0}=\scal{\nabla f_i(x)}{X(x)}.
$$
For $x\in \Gamma$, we have that, since $f$ is lower--semicontinuous, 
$f(x)=f^-_{\nu_\Gamma}(x)$ (with $\nu_\Gamma$ unit normal vector of $\Gamma$); 
since we are assuming (\ref{posScal}),
\begin{align*}
\frac{d}{dt}f(\Phi_t(x))_{|t=0} =&\lim_{h\to 0} \frac{f(\Phi_h(x))-f^-_{\nu_\Gamma}(x)}{h}
=\lim_{h\to 0} \frac{f^-_{\nu_\Gamma}(\Phi_h(x))-f^-_{\nu_\Gamma}(x)}{h} \\
=&\scal{\nabla f^-_{\nu_\Gamma}(x)}{X(x)}
\end{align*}
and this completes the proof.
\end{proof}

Having given first variation formulae, we discuss the stability of isoperimetric sets.
In a discontinuous density setting the latter has to be understood as follows: if the variation $\Phi_t$ is volume 
preserving, $V(t)=V(0)$ for all $t>0$, then perimeter has to increase, i.e. 
$P(0)\leq P(t)$. This means that 
if $P$ is differentiable at $0$ then $P'(0)\geq 0$. This condition can be
strengthened to $P'(0)=0$ when ${\rm spt} X\cap \Gamma=\emptyset$: in fact, in this case we have differentiability of $P(t)$ also when considering the vector field $-X$ instead of $X$. As a consequence we will obtain 
necessary conditions on the isoperimetric set at points far from $\Gamma$ and at points on $\Gamma$, as summarised in the 
following Propositions \ref{teofresnel} and \ref{CorTangential}.

If we assume some regularity on the isoperimetric set, then the first variation of the perimeter can
be rewritten in a different way. Far from $\Gamma$, standard regularity holds \cite{Mor03Reg}. 
The additional regularity we have to require is the $\Hnmu$--almost everywhere
$\ci^2$ regularity of $\Gamma$ and the regularity of traces 
of ${\cal F}E\cap \Omega_i$ on $\Gamma$, allowing an integration by parts in (\ref{firstVarPer})
in order to eliminate the term $\diver_\Sigma X$. This operation can be done if we assume the following:
\begin{description}
\item[(r1)] for every $i=1,\ldots,k$ the part of the boundary
$\Sigma_i={\cal F}E\cap \Omega_i$ is an $(N-1)$-rectifiable set with an $(N-2)$-rectifiable boundary 
$S_i=\partial \Sigma_i$ of finite $\Haus^{N-2}$ measure,
i.e. $\Sigma_i$ is a normal $(N-1)$--current: for $x\in S_i$, we denote
by $\nu_{S_i}(x)$ the unit normal vector orthogonal to $S_i$, contained in $T_x\Sigma_i$ and with $\nu_{S_i}$
pointing outside $\Sigma_i$;
\item[(r2)] $\Sigma_\Gamma={\cal F}E\cap \Gamma$, which is also an $(N-1)$--rectifiable 
set, has $(N-2)$--rectifiable boundary $\sigma=\partial \Sigma_\Gamma$ with
finite $\Haus^{N-2}$ measure, i.e. $\Sigma_\Gamma$ is a normal 
$(N-1)$--current: we denote by $\nu_\sigma$ the unit normal vector orthogonal to $\sigma$, contained
in $T\Sigma_\Gamma$ and pointing outside $\Sigma_\Gamma$.
\end{description}
In this case the first variation of the perimeter can be rewritten as
\begin{align}
\nonumber
P^\prime(0)
=&\int_{\Sigma}\scal{\nabla\psi^-}{\nu_\Sigma} uf^-d\Hnmu -(N-1) \int_{\Sigma} H_\Sigma f^- u d\Hnmu \\ 
&+\int_\sigma f^- \scal{X}{\nu_\sigma} d\Haus^{N-2} +\sum_{i=1}^k \int_{S_i}f_i 
\scal{X}{\nu_{S_i}}d\Haus^{N-2}.
\label{reg_stationary}
\end{align}
For the densities we shall consider in Section \ref{SecPCDens}, 
{\bf (r1)} and {\bf (r2)} are satisfied, so we shall not enter into further details and
simply say, if they hold true, that $E$ has {\em regular trace} on $\Gamma$.

We now give some necessary condition for the isoperimetric set; first of all, as
already mentioned in Remark \ref{rmkMainProperties}, the generalised mean curvature 
\begin{equation}\label{teoConstCurv}
H_\psi(x) =(N-1)H_\Sigma(x)-\scal{\nabla \psi(x)}{\nu_\Sigma(x)}
\end{equation}
is equal to a constant $H_0$ for any $x\in \Sigma\setminus \Gamma$.

Moreover, we have the next result, giving a sort of counterpart of the celebrated Snell refraction law of optics. 
It follows from a well--known fact, that the tangent cone to
an isoperimetric set must be perimeter--minimising without volume
constraint (see for instance \cite[Section 9]{Mor08Geo}). Anyway, we give the proof for reader's convenience.

\begin{prop}\label{teofresnel}
In the same notation as Proposition \ref{PropVarVolPerPC}, let 
$E$ be an isoperimetric set with regular traces on $\Gamma$; assume that
$\Sigma$ passes through $\Gamma$ transversally at a point
$x\in S_i\cap S_j$ where $S_i,S_j$ and $\Gamma$ have tangent
spaces, i.e. the vectors $\nu_{S_i}$, $\nu_{S_j}$ and $\nu_\Gamma$
 exist. Then the vector $f_i(x)\nu_{S_i}(x)+f_j(x)\nu_{S_j}(x)$ is
 parallel to $\nu_\Gamma(x)$, that is 
$$
f_i(x)\scal{\nu_{S_i}(x)}{\nu}=-f_j(x)\scal{\nu_{S_j}(x)}{\nu}
$$
for any $\scal{\nu}{\nu_\Gamma}=0$.
In particular the Snell refraction law
\begin{equation}\label{eqfresnel}
\frac{\cos\al_+}{\cos\al_-}=\frac{f^-_{\nu_\Gamma}(x)}{f^+_{\nu_\Gamma}(x)}
\end{equation}
holds. Here, $\al_+$ and $\al_-$ are, respectively, the two angles at $x$
between $\Gamma$ and $\Sigma$ from the two opposite sides of $\Gamma$ (see Figure
\ref{fresnel}). 
\end{prop}

\begin{figure}[h!tbp]
\begin{center} 
\scalebox{1}{ 
\input{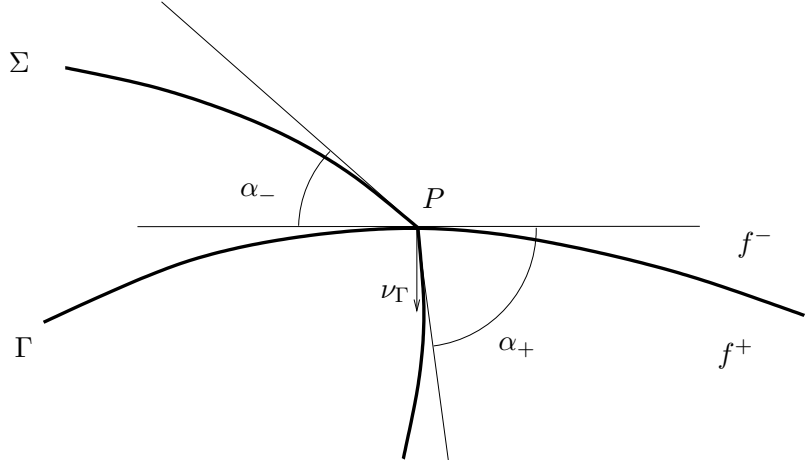} 
} 
\end{center}
\caption{Snell refraction law.}
\label{fresnel} 
\end{figure} 

\begin{proof}
First of all, we notice that $\Sigma\setminus \Gamma\neq \emptyset$, since
otherwise $\Sigma_i=\emptyset$ for any $i$ and then $S_i=\emptyset$.
Let $x_1:=x$ be as in the statement; consider
a second point $x_2\in \Sigma\setminus \Gamma$ and a sequence of
 functions $\varrho^r_{h,i}$, $i=1,2$ such that
$$
\varrho^r_{h,i}(y)\to \frac{1}{\omega_{N-1}r^{N-1}} \chi_{B(x_i,r)}(y),
 \quad i=1,2 
$$
and then consider the vector field
$$
X_h^r(y)= \varrho_{h,1}^r(y)\nu_1+c \varrho_{h,2}^r(y)\nu_2
$$ 
for some unit vectors $\nu_1$ and $\nu_2$ to be chosen in a convenient way. 
The constant $c$ has to be determined by requiring that $V'(0)=0$;
using \eqref{reg_stationary} and taking into account the continuity of the functions $f_i$ in $\overline \Omega_i$,
as $h\to+\infty$ we get
\begin{align*}
0\leq&
\frac{c}{\omega_{N-1}r^{N-1}}\int_{\Sigma\cap B(x_2,r)}
 \scal{\nabla\psi^-}{\nu_\Sigma}\scal{\nu_2}{\nu_\Sigma}f^-d\Hnmu \\
&+\frac{1}{\omega_{N-1}r^{N-1}}\int_{S_i\cap B(x_1,r)}f_i
 \scal{\nu_1}{\nu_{S_i}}d\Haus^{N-2} +\\
&+\frac{1}{\omega_{N-1}r^{N-1}}\int_{S_j\cap B(x_1,r)}f_j \scal{\nu_1}{\nu_{S_j}}d\Haus^{N-2}.
\end{align*}
Multiplying the previous equation by $r$ and taking the limit as $r\to
 0$, we obtain 
\begin{equation}\label{blowTang}
\scal{\nu_1}{f_i(x_1)\nu_{S_i}(x_1)+f_j(x_1)\nu_{S_j}(x_1)}=
f_i(x_1)\scal{\nu_1}{\nu_{S_i}(x_1)}+f_j(x_1)\scal{\nu_1}{\nu_{S_j}(x_1)}\geq 0.
\end{equation}
We are assuming that $\Gamma$ has a tangent hyperplane at $x_1$, and then there exists a
normal vector $\nu_\Gamma(x_1)$ with the usual agreement that
$f^-_{\nu_\Gamma(x_1)}=\min\{f_i(x_1),f_j(x_1)\}$ and $f^+_{\nu_\Gamma(x_1)}=\max\{f_i(x_1),f_j(x_1)\}$.
Recalling Proposition \ref{PropVarVolPerPC}, we have an admissible variation for any $\nu_1$ with the condition
$\scal{\nu_1}{\nu_\Gamma(x_1)}<0$. The validity of (\ref{blowTang}) for any such $\nu_1$ implies then
that the vector $f_i(x_1)\nu_{S_i}(x_1)+f_j(x_1)\nu_{S_j}(x_1)$
is a positive multiple of $\nu_\Gamma$ and 
$$
\scal{\nu_1}{f_i(x_1)\nu_{S_i}(x_1)+f_j(x_1)\nu_{S_j}(x_1)}=0, 
$$
for each $\scal{\nu_1}{\nu_\Gamma}=0$; 
this is the Snell law, since by transversality of $\Sigma$ at $x$, we can take 
$\nu_1$ as the unique vector (up to a sign) orthogonal to $\nu_\Gamma$ and to $T_{x_1}(S_i\cap S_j)$.
\end{proof}

Proposition \ref{teofresnel} will be of crucial importance
in Section \ref{SecPCDens}, where we will consider piecewise constant densities. Intuitively, Proposition  \ref{teofresnel} says that, when $\partial E$ crosses $\Gamma$ transversally, a corner is formed according to Snell law. Notice also that the angle $\al_+$ must be greater than $\arccos\big( f^-_{\nu_\Gamma}(x)/f^+_{\nu_\Gamma}(x)\big)$ and less than the supplementary of this angle. We will sometimes use \eqref{eqfresnel} in an equivalent form, where the cosines of the angles $\al_+,\al_-$ are substituted by the sines of the complementary angles, i.e. the ones formed by the normals to $\Gamma$ and $\Sigma$.

When the crossing between $\Sigma$ and $\Gamma$ is not transversal and $\Sigma$ touches $\Gamma$ in
the region $\sigma$ defined in {\bf (r2)}, we have the following result.

\begin{prop}\label{CorTangential}
Let $E$ be an isoperimetric set with regular trace on $\Gamma$; assume that for 
$x\in \sigma\cap S_i$, $\nu_\sigma$, $\nu_{S_i}$ and $\nu_\Gamma$ exist, i.e. 
$\sigma, S_i$ and $\Gamma$ have tangent planes. If 
$f_i(x)>f^-_{\nu_\Gamma}(x)$, then
\begin{equation}\label{SnellTan}
\scal{\nu_\sigma(x)}{\nu_{S_i}(x)}\geq-\frac{f^-_{\nu_\Gamma}(x)}{f_i(x)},
\end{equation}
that is, $\Sigma_i=\Sigma\cap \Omega_i$ meets $\Gamma$ at an angle at most 
$\arccos\Big(\frac{f^-_{\nu_\Gamma}(x)}{f_i(x)}\Big)$; 
if $f_i(x)=f^-_{\nu_\Gamma}(x)$, then $\nu_\sigma(x)=-\nu_{S_i}(x)$, that
is, $\Sigma_i$ is tangential to $\Gamma$.
\end{prop}
 
\begin{proof}
The proof is similar as for Proposition \ref{teofresnel}; setting again $x_1=x$ and considering
$$
X^r_h(x)=\varrho^r_{h,1}(x)\nu_1 +c\varrho^r_{h,2}(x)\nu_2
$$
with $\varrho^r_{h,i}$, $i=1,2$, as before, by taking the limits as $h\to +\infty$
and $r\to 0$ one gets
\begin{equation}\label{nunu}
f^-_{\nu_\Gamma}(x_1)\scal{\nu_1}{\nu_{\sigma}(x_1)}+f_i(x_1)\scal{\nu_1}{\nu_{S_i}(x_1)}\geq
 0.
\end{equation}
This inequality holds for any $\nu_1$ with $\scal{\nu_1}{\nu_\Gamma}>0$; if $f_i(x_1)>f^-_{\nu_\Gamma}(x_1)$,
then by taking $\nu_1=\nu_{S_i}(x_1)$ we obtain (\ref{SnellTan}). Otherwise, if
$f_i(x_1)=f^-_{\nu_\Gamma}(x_1)$, then (\ref{nunu}) reduces to
$$
\scal{\nu_1}{\nu_{\sigma}(x_1)+\nu_{S_i}(x_1)}\geq 0
$$
for any $\scal{\nu_1}{\nu_\Gamma}<0$. Since $\scal{\nu_\sigma}{\nu_\Gamma}=0$, 
the only possibility is that $\nu_{\sigma}(x_1)+\nu_{S_i}(x_1)=0$.
\end{proof}

\begin{rmk}
The previous proofs can be summarised and generalised as follows. 
Let us assume that the surface $\Gamma$ admits a unique tangent cone at the point
 $x_0$ and let us
denote by $O_i$ the blow-up at $x_0$ of the region
$\Omega_i$. If the isoperimetric set passes through $\Gamma$ at $x_0$,
 then there holds
$$
\scal{f_i\nu_{S_i}+f_j\nu_{S_j}}{\nu}\geq 0
$$
for any direction $\nu$ contained in $O_i$ where $f_i$ is minimal among
the $f_j$'s such that $x_0\in \overline \Omega_j$. This means that if
$\Gamma$ has a tangent plane at $x_0$, Proposition \ref{teofresnel}
 holds since $\nu$ can range over a half-space;
moreover,  $O_i$ cannot contain more than a half-space, otherwise we
would have
$$  
f_i\nu_{S_i}+f_j\nu_{S_j}=0
$$
which is never possible if $f_i\neq f_j$. Finally, if $O_i$ is less then
a half-space, then the conclusion is that $f_i\nu_{S_i}+f_j\nu_{S_j}$
is contained in a sector.
\end{rmk}

\section{Some piecewise constant densities}
\label{SecPCDens}

In this section we consider some particular piecewise regular densities on Euclidean space, 
mainly piecewise constant densities. 
Section \ref{SecHalfspace} focuses on a density in $\R^N$ taking value one in a half-space, 
and defined in the other half-space by a certain  real function with values greater than one. 
In Section \ref{SecStrip} we study the case of density one in a planar strip and a greater constant outside the strip. Finally, in Section \ref{SecBalls} we deal with a planar density taking value one outside a ball and a different constant value inside.

\subsection{The half-space}
\label{SecHalfspace}
In this section we will focus our attention on the half-space density in $\R^N$ defined by
$$
f(x)=1+(\lambda(x_N)-1)\chi_{\{x_N>0\}}(x)
$$
where $\lambda:\R\to [1,+\infty)$ is a bounded measurable 
function with $\lambda(x_N)>1$ for $x_N>0$. 
Observe that the weighted volume and perimeter coincide with the Euclidean ones in $\{x_N\leq 0\}$. 
We give the following preliminary result.

\begin{lemma}\label{constrLemma}
Let $E$ be a bounded set; then there exists a set $E'$ entirely contained in $\{x_N<0\}$ 
with $|E'|=|E|_f$ and $P(E')\leq P_f(E)$, with equality if and only if almost all of $E$ 
is contained in $\{x_N\leq 0\}$. 
\end{lemma}

\begin{proof}
Given a bounded set $E\subset \Rn$, we write $E_1:=E\cap
\{x_N<0\}$ and $E_2:=E\cap \{x_N\geq 0\}$; we claim that if $|E_2|>0$, then
there exists a set in $\{x_N<0\}$ with the same volume and strictly 
less perimeter than $E$. The volume of $E$ is given by
$$
|E|_f =|E_1|+\int_{E_2} \lambda(x_N)dx.
$$ 
We define the transformation $F:\Rn\to \Rn$ by
$F(x)=F(x^\prime,x_N)=(x^\prime,F_N(x_N))$, with
$$
F_N(x_N)=\left\{\begin{array}{ll}
x_N & \mbox{if } x_N<0 \\
\mbox{ } & \\
{\displaystyle \int_0^{x_N} \lambda(t)dt} & \mbox{if } x_N\geq 0.
\end{array}
\right.
$$
For this transformation we have that $JF$ is the identity $I_N$ on
$\Rn$ if $x_N<0$, while for $x_N> 0$ we have 
$$
J F(x)=\left(
\begin{array}{cc}
I_{N-1} & 0\\
0 & \lambda(x_N)
\end{array}
\right).
$$
This means that the set $\hat E:=F(E_2)$ has the
Lebesgue measure
\begin{equation}\label{Fvol}
|\hat E|=|E_2|_f.
\end{equation}
Regarding the Euclidean perimeter of $\hat E$, by Lemma \ref{lemmaGiusti} 
\begin{eqnarray}\label{Fper}
\nonumber
P(\hat E, \{x_N>0\})&=& 
\int _{{\cal F}E\cap\{x_N>0\}}
\sqrt{\lambda(x_N)^2(\nu_1^2+\ldots \nu_{N-1}^2)+\nu_N^2} d\Hnmu\\
&\leq&
\int _{{\cal F}E\cap\{x_N>0\}} \lambda(x_N) d\Hnmu
\end{eqnarray}
with equality if for $\Hnmu$--a.e. $x\in{\cal F}E$ we have
$\nu_N=0$. This cannot happen unless either $|E_2|=0$ or $E_2$ is 
a vertical cylinder; but in this second case, $E$ would not be bounded. 
Since $E$ is bounded, it is contained in a strip $\{|x_N|<r\}$ for some positive $r>0$. 
By translating down by $r$ in the $N$--th direction 
the set $F(E)$, we can construct a set contained in $\{x_N<0\}$, with
the same volume (in view of \eqref{Fvol}) as $E$ and less or equal perimeter (\ref{Fper}), which proves the claim.
\end{proof}

\begin{theo}\label{Serapioni}
For given volume, an isoperimetric set exists and is a round ball contained in the half--space
$\{x_N\leq 0\}$.
\end{theo}
\begin{proof}
We may consider a minimising sequence $F_h$ of smooth sets, that is 
$F_h$ with $|F_h|_f=v$ and $P_f(F_h)\leq \alpha+\frac{1}{h}$ where
$$
\alpha =\inf \{ P_f(E): |E|_f=v\}.
$$
If $F_h$ is bounded, we set $E_h=F_h$; otherwise, 
if $F_h$ is not bounded, there exists $t>0$ such that $\Hnmu(\partial F_h \setminus \overline B_t) < \tfrac 1 h$ and
$$
\Hnmu (F_h\cap\partial B_t)=\Hnmu(\partial(F_h\setminus B_t)\cap\partial B_t)=\Hnmu(\partial(F_h\cap B_t)\cap\partial B_t)<\tfrac 1h, 
$$
where we used the fact that the first two equalities hold for a.e. $t>0$. By the Euclidean
isoperimetric inequality, we have that 
$$
|F_h\setminus B_t|<\frac{c_1}{h^\frac{N}{N-1}}.
$$
We can then consider a ball $B_r(x_0)$ with $r$ chosen in such a way that
$$
|B_r(x_0)|= |F_h\setminus B_t|_f < \frac{c_2}{h^\frac{N}{N-1}}
$$
and $x_0$ such that $B_r(x_0)$ is contained in $\{x_N<0\}$ and not intersecting $B_t$. 
Notice that $r\leq \tfrac {c_3}{h^{1/N-1}}$ and so the set
$$
E_h=(F_h\cap B_t\}) \cup B_r(x_0)
$$
has the properties $|E_h|_f=v$ and
\begin{align*}
P_f(E_h)=&P_f(F_h\cap B_t\})+\int_{\partial(F_h\cap B_t)\cap\partial B_t} \la\,d\Hnmu  + P(B_r(x_0))\\
\leq&
P_f(F_h)+\tfrac {c_4}{h}
\end{align*}
that is $E_h$ is also a minimising sequence with bounded sets. 

By Lemma~\ref{constrLemma}, we may assume that $E_h$ is contained in $\{x_N<0\}$, where the
density $f=1$. By the standard isoperimetric inequality, round balls are isoperimetric. By
a second application of Lemma~\ref{constrLemma}, round balls in $\{x_N\leq 0\}$ are uniquely isoperimetric.
\end{proof}

\begin{rmk}
When $\lambda:\R^N\to[1,+\infty)$ does not depend only on the last coordinate,
then it is no longer true in general that round balls in $\{x_N\leq 0\}$ are isoperimetric.
For instance, for the density
\begin{equation}\label{DenBall}
f(x)=1+\mu \chi_{B(2e_N,1)}(x)
\end{equation} 
with $\mu>0$, $B(2e_N,1)$ is isoperimetric. 
We shall 
discuss in Section \ref{SecBalls} the case of the density (\ref{DenBall}) in the plane $\R^2$.
\end{rmk}

\subsection{The strip}\label{SecStrip}
In this section we shall consider the ``strip'' density in $\rr^N$ given by
$$
f(x)=\left\{
\begin{array}{ll}
1 & \mbox{if } |x_N|\leq 1 \\
\lambda\in\R & \mbox{if } |x_N|>1
\end{array}
\right.
$$
with $\lambda>1$. 

\begin{rmk}
For the density
$$
f(x)=\left\{
\begin{array}{ll}
\lambda & \mbox{ if }|x_N|<1 \\
1 & \mbox{ if } |x_N|\geq 1
\end{array}
\right.
$$
isoperimetric sets exist and are round balls outside the slab. Indeed, Lemma~\ref{constrLemma}
shows that any set $E$ may be replaced by a set outside the slab with no more perimeter, with equality if and only
if the boundary of $E$ is vertical inside the slab. But if this fact occurs, the original set $E$ will be an unbounded vertical cylinder or will not have constant curvature,
and so cannot be isoperimetric. This allows us to consider only sets contained in $\{|x_N|>1\}$.
\end{rmk}

We start with an existence result for the planar case.

\begin{prop}\label{PropExistN2}
For the above strip density in the case $N=2$, an isoperimetric set exists for any prescribed volume and is connected.
\end{prop}

\begin{proof}
Take a minimising sequence $E_h$ of smooth sets with $|E_h|_f=v$ and each $E_h$ consisting of finitely many
components. By Lemma \ref{constrLemma} we may assume that any connected component of $E_h$ 
either has non--empty intersection with the lines
$\{y=1\}$ and $\{y=-1\}$ or is contained in 
$\{-1\leq y\leq1\}$. 

Move horizontally (or also vertically, but remaining inside the strip, in case of a ball contained in it) the connected components of $E_h$ until they touch each other tangentially; using a cutting argument, replace $E_h$ with a connected one. 
Move horizontally the sets $E_h$ to have barycenter $(0,y_h)$; the sets $E_h$ are connected and so, using the diameter estimate (\ref{estdiamper}), they 
satisfy $|y_h|\leq P_f(E_h)+1$. This means that all the sets $E_h$ are contained in a big ball,
and a compactness argument implies the existence of an isoperimetric region: were it not connected, we could use one of the above ``components translation'' argument to diminish perimeter.
\end{proof}

We can also state the following result that holds in any dimension $N$, under the assumption
that the isoperimetric set exists.

\begin{prop}\label{WSteiner}
Let $E$ be an isoperimetric set for given volume $v>0$, for the above strip density. 
Then, there exists an isoperimetric set for volume $v$ which is rotationally symmetric with respect to the $x_N$-axis. 
Moreover, the intersection of $E$ with almost every horizontal hyperplane is a round ball, 
$E$ has regular trace on $\{|x_N|=1\}$ and $E$ is either entirely contained in $\{-1< x_N<1\}$ or 
it will touch simultaneously both sets $\{x_N\geq1\}$ and $\{x_N\leq-1\}$.

\end{prop}

\begin{proof}
The argument is based on Schwarz symmetrisation (see \cite[\S~3.2]{Ros05The}), 
which can be applied in this setting; 
define $E=E_1\cup E_2\cup E_3 \cup S_1\cup S_{-1}$ with
$$
E_1=E\cap \{x_N>1\}, \quad E_2=E\cap \{-1<x_N<1\}, \quad E_3=E\cap \{x_N<-1\} 
$$
and 
\begin{equation}\label{S}
S_{\pm1}=\{x\in \Rn: x_N=\pm1, |\chi_{E,e_N}^+(x)-\chi_{E,e_N}^-(x)|=1 \}.
\end{equation}
We also define the sets
$$
S_1^\pm=\{x\in \Rn: x_N=1, \chi_{E,e_N}^\pm=1\},\quad
S_{-1}^\pm=\{x\in \Rn: x_N=-1, \chi_{E,e_N}^\pm=1\}.
$$
Notice that
\begin{align*}
P_f(E)=&\Hnmu(S_1\cup S_{-1})+\Hnmu({\cal F}E\cap \{-1<x_N<1\})+\\
&+\lambda\Hnmu({\cal F}E\cap \{x_N>1\}+
\lambda\Hnmu({\cal F}E\cap \{x_N<-1\}).
\end{align*}
We can now consider the Schwarz-symmetrised
set of $E$, 
i.e. the set $E^*$ such that $E^*\cap\{x_N=\bar y\}$ is a round ball centered at $(0,\dots,0,\bar y)$ with measure equal to $\Hnmu(E\cap\{x_N=\bar y\})$. One has 
\begin{equation}\label{symmetrized}
|E^*|=|E|, \qquad P(E^*)\leq P(E)
\end{equation}
and almost every horizontal slice of $E$ is a round ball in case $P(E^*)= P(E)$ (see \cite[\S~1]{CCF} or \cite{talenti}).
We notice that (\ref{symmetrized}) holds also on the sets $E_i$, $i=1,2,3$.
As
\begin{align*}
\Hnmu(S^*_{\pm 1}) =\:&\Hnmu\big( S^{*+}_{\pm 1}\Delta S^{*-}_{\pm 1}\big)\
=\ |\Hnmu(S^+_{\pm 1})-\Hnmu(S^-_{\pm 1})|  \\
=\:&
\left|
\int_{\R^{N-1}} (\chi_{E,e_N}^+ -\chi_{E,e_N}^-)d\Hnmu
\right| \\
\leq\: & 
\int_{\R^{N-1}} |\chi_{E,e_N}^+-\chi_{E,e_N}^-| d\Hnmu
= \Hnmu(S_{\pm1}),
\end{align*}
where $S^*_{\pm 1}$ and $S^{*\pm}_{\pm 1}$ are defined as in (\ref{S}) with $E^*$ instead of $E$, 
we can conclude, since $f$ is constant along horizontal directions, that
\begin{equation}\label{eq:Schw}
|E^*|_f=|E|_f\qquad\text{and}\qquad P_f(E^*)\leq P_f(E). 
\end{equation}
Since $E$ is isoperimetric, we have the equality in \eqref{eq:Schw}, 
and so $E^*$ is an isoperimetric solution for volume $v$, 
which is in fact rotationally symmetric with respect to the $x_N$-axis. 
Moreover, we also deduce that the intersection of $E$ 
with almost every horizontal hyperplane is a round ball. 

The regularity of the trace of $E$ on $\{|x_N|=1\}$ follows since the intersections
of the isoperimetric set with the planes $\{x_N=1\pm \varepsilon\}$ (similar argument on the jump part
$\{x_N=-1\}$) are round balls converging in $L^1(\R^{N-1})$ as $\varepsilon \to 0$ 
(see for instance Giusti \cite[Theorem 2.11]{Giu84Min}); 
therefore the traces on the jump set are round balls.

For the last sentence, if $|E\cap \{x_N>1\}|>0$ and
$d(E,\{x_N<-1\})>0$, we can move down $E$ using Lemma~\ref{constrLemma}, 
and modify $E$ into a set $\tilde E$ with same volume and less perimeter.
%
\end{proof}
%

\begin{rmk}
Notice that Proposition \ref{WSteiner} implies that the isoperimetric
problem for the ``strip'' density is essentially two-dimensional, meaning 
that the boundary of the isoperimetric set is given by the rotation
on the vertical axis of the graph of a generating function $h=h(x_N)$. 
Let us stress that this constitutes a different problem with respect to the ``strip'' case with $N=2$, since 
now the isoperimetric problem to be considered has different
densities on volume and perimeter, more precisely
\begin{align*}
\tilde V_f(E)=&\omega_{N-1}\int_\R f(t)h^{N-1}(t)dt, \\
\tilde P_f(E)=&(N-1)\omega_{N-1} \int_\R f(t)h(t)^{N-2}\sqrt{1+h'(t)^2}dt.
\end{align*}
\end{rmk}


We will now focus on the case $N=2$. 
In this situation, we know that isoperimetric sets exist, 
are connected and that almost all horizontal slices are intervals
(from Propositions \ref{PropExistN2} and \ref{WSteiner}).
Furthermore, Snell law \eqref{eqfresnel} and constant geodesic curvature condition 
are satisfied. 
Therefore, we have that the boundary of any isoperimetric set $E$ is a piecewise regular curve $\gamma$ 
with regular pieces consisting of either arcs $\gamma_i$ of some circles with the same curvature
or line segments contained in $\{|y|=1\}$. 
Moreover, the part of $\gamma$ contained in $\{(x,y): y>1\}$ is at most one arc of a circle; 
the same holds in $\{(x,y): y<-1\}$, while there can be two arcs inside the strip. 
This allows to classify the isoperimetric candidates which are rotationally symmetric in this case. 

\begin{lemma}\label{lemstripposs}
For the strip density for $N=2$ the only isoperimetric candidates 
with vertical reflective symmetry (see Figure \ref{strip1}) are, up to density-preserving isometries: 
\begin{itemize}
\item[(i)] balls contained in the strip;  
\item[(ii)] pieces of the strip bounded left and right by two semicircles, and top and bottom by two segments on $\{|y|=1\}$; 
\item[(iii)] sets bounded by three circular arcs with the same radius 
and a segment contained in $\{y=-1\}$. One of these arcs is contained in $\{y>1\}$, 
and the other two are contained in the strip, 
meeting the segment tangentially and the other arc according to the Snell law
(\ref{eqfresnel}). 
The segment can degenerate to a single point; in this case, the set is a circle with radius 2 and center on the line $\{y=1\}$;
\item[(iv)] sets bounded by four circular arcs with the same radius: two of 
them contained in the strip; the other two 
horizontally symmetric, one contained in $\{y>1\}$ and another contained in $\{y<-1\}$, 
satisfying the Snell law.
\end{itemize}

\begin{figure}[htbp]
\begin{center} 
\scalebox{1}{ 
\input{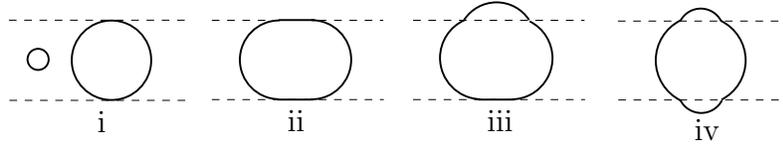} 
} 
\end{center}
\caption{The candidate isoperimetric sets.}
\label{strip1} 
\end{figure}
\end{lemma}

\begin{proof}
Let us briefly sketch how to restrict to such candidates. If an isoperimetric set is completely contained in the strip it is easy to reduce to candidates (i) and (ii). Otherwise, by symmetry we can suppose there is an arc in $\{y>1\}$ meeting the strip; 
Lemma \ref{constrLemma} allows us to discard balls outside the strip. 
Unless the contact angle is not $\arccos 1/\la$, the arc is continued (starting from each of its endpoints according to Snell law) by two arcs inside the strip: they have to meet the line $\{y=-1\}$, or we could use construction of Lemma \ref{constrLemma}.
If the meeting is tangential there are two possibilities: the boundary can be continued by a segment and then a tangential arc in the strip (case (iii)) or immediately by an arc in the strip. In this last case there would be an arc inside the strip tangent to $\{y=-1\}$ plus another arc in $\{y>1\}$: they would have the same curvature and two common endpoints, so they must form a unique complete circle or be symmetric with respect to the line $\{y=1\}$. This contradicts the Snell law unless the meeting angle with $\{y=1\}$ is $\tfrac\pi2$, i.e., the set is a ball of radius 2. There is a priori another possibility (see Figure \ref{figcasev}): the meeting with $\{y=-1\}$ is tangential and the boundary is continued by a segment and then an arc in $\{y<-1\}$ with angle $\arccos\tfrac{1}{\la}$. We will exclude this possibility in Step 1 below.

If the meeting with $\{y=-1\}$ is not tangential there must be an arc in $\{y<-1\}$: this is case (iv). 
Step 2 below will prove that this set must be also symmetric with respect to the horizontal axis $\{y=0\}$.

If the upper arc meets $\{y=1\}$ with angle $\arccos \tfrac 1\la$ there are two possibilities for each of its endpoints: it is continued by a segment and then by a tangential arc, or immediately by a tangential arc. Again, these arcs must meet the line $\{y=-1\}$: if the meeting is not tangential there should be an arc in $\{y<-1\}$. This configuration is symmetric to the one 
of Figure \ref{figcasev}, which  we are going to exclude in Step 1. 
If the meeting is tangential, the arc's radius must be 1 and one must have also a segment on $\{y=-1\}$ and, possibly, another arc on $\{y<-1\}$ with meeting angle $\arccos \tfrac 1\la$ as in Figure \ref{fig1or2tangarcs}: this set cannot be isoperimetric 
as shown in Step 3.

A little problem could be given by the upper arc meeting $\{y=1\}$ with angle $\pi-\arccos \tfrac1\la$ (see Figure \ref{figline} in Theorem \ref{theoline}): the continuation would be a segment on $\{y=1\}$ (no inner arc is allowed by Snell law). This set would be entirely contained in $\{y\geq 1\}$, thus not being isoperimetric.

{\em Step 1.} 
We will show that the configuration of Figure \ref{figcasev} is geometrically impossible. This supposed
configuration consists of
two segments in $\{y=-1\}$ (each of them possibly reducing to a single point) and four circular arcs with the same radius; two of them vertically symmetric, contained in the strip and tangent to $\{y=-1\}$ but not to $\{y=1\}$; the other two contained in each component of $\rr^2\setminus\{|y|<1\}$, all satisfying the Snell law. In particular, the arc in $\{y<-1\}$ meets the strip with angle $\arccos\tfrac1\la$.

\begin{figure}[htbp]
\begin{center} 
\scalebox{1}{ 
\input{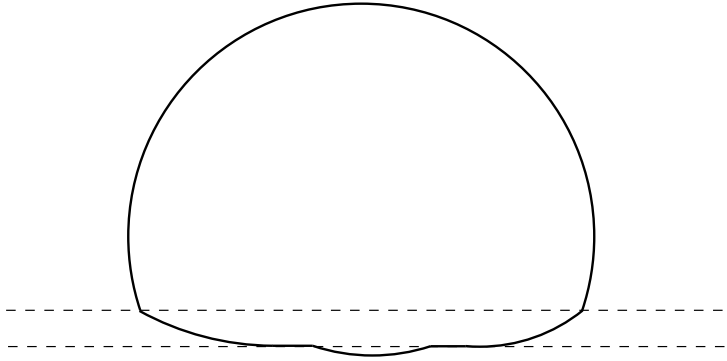} 
} 
\end{center}
\caption{A not feasible configuration.}
\label{figcasev} 
\end{figure}

We can think of this configuration as obtained from a configuration (iii) by adding the arc in $\{y<-1\}$; 
we will show that this arc is too long to fit.
The chord associated with this arc rests upon the segment on $\{y=-1\}$ which appears in (iii). This chord has length $2R\sqrt{1-\tfrac{1}{\la^2}}$, $R\in[1,+\infty[$ being the common radius of the arcs. 
We will get a contradiction if we show that the segment of configuration (iii) is shorter. Indeed, its length is equal to
\begin{eqnarray}
&&2R\left( \sqrt{1-\tfrac{1}{\la^2}\big( 1-\tfrac2R\big)^2}-\sqrt{1-\big( 1-\tfrac2R\big)^2}\right)\nonumber\\
&=& 2R\left( \sqrt{\big(1-\tfrac{1}{\la^2}\big)+\tfrac{1}{\la^2}\de}-\sqrt{\de}\right)\nonumber\\
&\leq& 2R\left( \sqrt{1-\tfrac{1}{\la^2}}+\big(\tfrac{1}{\la}-1\big)\sqrt{\de}\right),
\label{deltaappears}
\end{eqnarray}
where we put $\de:=\tfrac4R-\tfrac4{R^2}\geq0$ (since $R\geq 1$) and used the inequality $\sqrt{a+b}\leq\sqrt a+\sqrt b$. The last term in \eqref{deltaappears} is strictly less than $2R\sqrt{1-\tfrac{1}{\la^2}}$ unless $\de=0$, i.e. $R=1$. The latter case $R=1$ will be discarded in the following Step 3.

{\em Step 2.} 
Here we show that the four--arcs configuration of type (iv) must be symmetric with respect to the horizontal
line $\{y=0\}$. As in Figure \ref{fig4nonsymm}, let $P$ be the center of the arc in $\{y>1\}$, and let $A,C$ be its endpoints on the line $\{y=1\}$. Similarly, $Q$ is the center of the lower arc and $B,D$ its endpoints. By $O$ we denote the center of the left arc inside the strip; $K$ belongs to this arc and is chosen so that $OK$ is horizontal. 

\begin{figure}[htbp]
\begin{center} 
\scalebox{1}{ 
\input{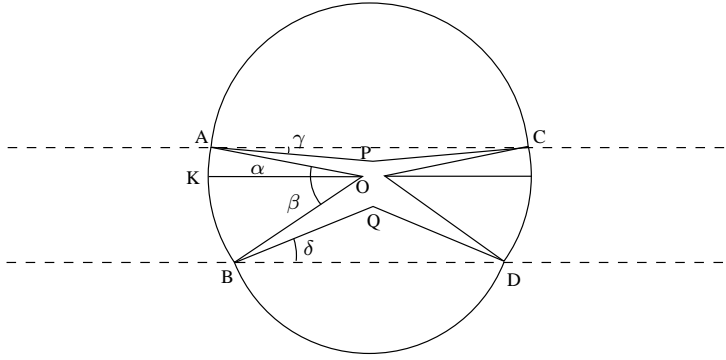} 
} 
\end{center}
\caption{A four-arcs configuration must be horizontally symmetric.}
\label{fig4nonsymm} 
\end{figure}

It will be convenient to set
$$
\al:=\widehat{AOK},\quad\be:=\widehat{BOK},\quad\g:=\widehat{CAP},\quad\de:=\widehat{DBQ}
$$
and $R:=AP=AO=BO=BQ$ is the common radius of the arcs. Notice that $R$ and $\al$ uniquely determine the configuration; indeed, since the height of the strip is 2,
\begin{equation}\label{eqperbeta}
R(\sin\al+\sin\be)=2,\quad\text{i.e.}\quad\sin\be=\tfrac2R-\sin\al.
\end{equation}
Moreover, $\al,\g$ are also the angles formed by the vertical line passing through $A$ and the two arcs meeting at $A$; similarly for $\be,\de$ at $B$. Thus the Snell law
\begin{equation}\label{eqpergammadelta}
\frac{\sin\al}{\la}=\sin\g,\qquad\frac{\sin\be}{\la}=\sin\de
\end{equation}
holds. When $\al$ is negative (i.e. $O$ lies above the line $\{y=1\}$) computations are the same.

Recall that the configuration must be symmetric with respect to a vertical line, which must contain $P$ and $Q$. 
Equivalently, $P$ and $Q$ must belong to such axis: they lie ``one above the other''. Equivalently, the first component of the oriented segment $OP=OA-PA$ is equal to the one of $OQ=OB-QB$, i.e. $R(\cos\al-\cos\g)=R(\cos\be-\cos\de)$. Taking into account \eqref{eqperbeta} and \eqref{eqpergammadelta}, we can exploit this condition to get
$$
\sqrt{1-s^2}-\sqrt{1-\tfrac{s^2}{\la^2}}=\sqrt{1-(\tfrac2R-s)^2}-\sqrt{1-\tfrac{1}{\la^2}(\tfrac2R-s)^2},
$$
where we have set $s:=\sin\al$. Let us study the function $\p:[-1,1]\to\R$ defined by $\p(s):=\sqrt{1-s^2}-\sqrt{1-\tfrac{s^2}{\la^2}}$. It is an even map and
$$
\p'(s)=-\frac{s}{\sqrt{1-s^2}}+\frac{s}{\la\sqrt{\la^2-s^2}}=s\:\frac{\sqrt{1-s^2}-\la\sqrt{\la^2-s^2}}{\la\sqrt{1-s^2}\sqrt{\la^2-s^2}}
$$
is negative for $s>0$ and positive for $s<0$. Thus $\p$ is even, increasing in $[-1,0[$ and decreasing in $]0,1]$. In order to have $\p(s)=\p\big( \tfrac2R-s\big)$ it necessarily holds that $s=\pm\big( \tfrac2R-s\big)$, whence $s=1/R$. This condition corresponds to $\al=\be=\arcsin\tfrac 1R$, which immediately implies symmetry with respect to $\{y=0\}$.

{\em Step 3.} 
We have to discard sets $E$ as in Figure \ref{fig1or2tangarcs} bounded by: one arc in $\{y>1\}$ and/or one in $\{y<-1\}$ meeting the strip with angle $\arccos\tfrac1\la$; one or two segments on $\{y=1\}$ and one/two on $\{y=-1\}$ (any of them can be a single point); two arcs in the strip meeting tangentially the segments. Notice that the radius of any of the four arcs is 1, since the inner ones are tangent to 
$\{|y|=1\}$. 

\begin{figure}[htbp]
\begin{center} 
\includegraphics[width=10cm]{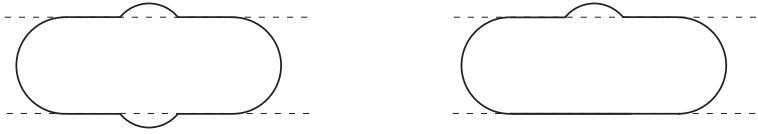}
\end{center}
\caption{A not minimising configuration.}
\label{fig1or2tangarcs} 
\end{figure}


Set $\al_\la:=\arccos\tfrac1\la$. Cut from $E$ the region enclosed by one of the arcs outside the strip. In this way we lose area $(\al_\la-\sin\al_\la\cos\al_\la)\la$. Since the arc disappears, but a segment is created, we lose perimeter $2\la\al_\la-2\sin\al_\la$. We restore area in the following way: ``enlarge'' $E$ horizontally until area is restored, for example by translating the right arc inside the strip on the right. Two segments (one on $\{y=1\}$ and one on $\{y=-1\}$) are created; their length $L$ must satisfy the restoring-area condition
$$
2L=(\al_\la-\sin\al_\la\cos\al_\la)\la=\la\al_\la-\sin\al_\la\,.
$$
This operation produces a growth of perimeter equal to $2L$. To show that we have decreased perimeter we must check that
$$
\la\al_\la-\sin\al_\la<2\la\al_\la-2\sin\al_\la
$$
or equivalently that 
$$
\sin\al_\la<\la\al_\la, 
$$
which is always satisfied.
\end{proof}

Next result assures that any isoperimetric set in this setting must be vertical reflective symmetric. 

\begin{prop}
\label{prop:simetrico}
Let $E$ be an isoperimetric region for the strip density. 
Then, $E$ has vertical reflective symmetry. 
\end{prop}

\begin{proof}
Consider the symmetrised set $E^*$ of $E$, 
which will be of one of the types described in Lemma~\ref{lemstripposs}. 
If $E^*$ is of type (i), then $E$ is entirely contained in the strip, 
and it follows trivially that it must be a ball, thus symmetric. 
If $E^*$ is of type (ii), then the part of $\partial E$ contained in $\{|y|=1\}$ 
consists of a segment in $\{y=1\}$ and another one in $\{y=-1\}$, both of equal length. 
The upper segment must be continued left and right by two tangential arcs of the same radius 
(since curvature is constant and Snell law holds). This already gives the symmetry of $E$ 
(we are also using that horizontal slices of $E$ are intervals).
Finally, if $E^*$ is of type (iii) or (iv), we proceed analogously: 
the part of $\partial E$ contained in $\{y>1\}$ is an arc of a circle, 
which must be continued by two arcs of the same radius inside the strip, 
forming identical angles 
(note that segments in $\{y=1\}$ are not allowed, since $E^*$ have not such segments); 
this yields the desired symmetry.
\end{proof}

\begin{rmk}\label{rem:symm1}
The vertical reflective symmetry of an isoperimetric set $E$ 
which produces a set $E^*$ of type (ii) after Schwarz symmetrisation 
also follows from the following argument: 
call $L>0$ the length of the top slice, and remove an interval of length $L$ 
from the right endpoint of each horizontal slice of $E$. The resulting set 
has the same volume and perimeter that the biggest ball inside the strip, 
so it must coincide with that ball. 
Adding now the removed intervals, we get that the original set $E$ is symmetric. 
On the other hand, in case that $E^*$ consists of a set of type (iii) or (iv), the portion 
of $\partial E$ in $\{y>1\}$ must be optimal for the free boundary problem (fixing the 
corresponding segment on $\{y=1\}$), so it must be an arc of a circle; 
for case (iv), an identical arc will appear in $\{y<-1\}$. 
As pieces of $\partial E$ contained in \{y=1\} are not allowed in these cases, 
the symmetry of $E$ follows since inside the strip we must also have the optimal 
configuration for the fixed boundary segments. 
\end{rmk}

\begin{rmk}\label{rem:symm2}
Recall that the horizontal slices of $E$ are intervals in view of Proposition \ref{WSteiner}. 
Consider the symmetrised set $E^*$, 
which will be one of those sets described in Lemma \ref{lemstripposs};
it is clear that the length of the horizontal slices of $E^*$ 
is a continuous positive function of $y$, 
and so the same property holds for $E$. 
This also suffices to conclude, by using \cite[Th.1.3]{CCF}, that $E$ has vertical symmetry.
\end{rmk}

In view of Lemma~\ref{lemstripposs} and Proposition~\ref{prop:simetrico}, 
we have completely classified the isoperimetric candidates for the strip density. 
We will now analyse numerically their behaviour. Our computations below lead to 
the following Conjecture \ref{conj:strip}. 
Theorem \ref{teostrip} proves everything except the elimination of sets of type (iv).

\begin{conjecture}
\label{conj:strip}
Given $\rr^2$ with the strip density, $1$ in the strip $\{|y|\leq 1\}$ and $\lambda>1$ outside,
there exists a value $v_0>\pi$ such that the isoperimetric sets are:
\begin{itemize}
\item[(a)] balls of type (i) for areas less than $\pi$;
\item[(b)] sets of type (ii) for values of the area in $[\pi,v_0]$;
\item[(c)] sets of type (iii) for areas greater than $v_0$.
\end{itemize}  
\end{conjecture}

It is not difficult to compare the perimeters 
of the above candidates (i), (ii), (iii) and (iv), for equal volumes. 
For a ball in the strip enclosing area $v\leq\pi$, 
its perimeter is given by $P_{i}(v)=2\,\sqrt{\pi\,v}$; 
and a set of type (ii) enclosing area $v\geq\pi$ has perimeter given by $P_{ii}(v)=v+\pi$. 

On the other hand, by using the Snell law and, 
for instance, computations in \cite[Prop. 2.1]{double}, 
sets of type (iii) can be parametrized in terms of the 
corresponding generalised mean curvature $h$. Thus, denoting 
by 
\begin{equation}\label{alphabetabis}
\beta:=\beta(h)=\pi-2\arcsin(\sqrt{h}),\quad 
\alpha:=\alpha(h)=\arccos\bigg(\frac{\cos\beta}{\lambda}\bigg),
\end{equation}
it can be checked that the perimeter and the area enclosed by such a set are given by
\begin{align*}
P_{iii}(h)=&\frac{2\,\lambda\,\alpha}{h}+\frac{4 \arcsin(\sqrt{h})}{h}+\frac{2\sin\alpha}{h}-\frac{4\,\sqrt{h(1-h)}}{h}, 
\\
A_{iii}(h)=\,&\frac{\lambda\,\alpha-\sin\alpha\:\cos\beta}{h^2}+
2\,\frac{\arcsin(\sqrt{h})-\sqrt{h(1-h)}}{h^2}\,
+\frac{4\,\sin\alpha}{h} -\frac{4\,\sqrt{1-h}}{\sqrt{h}}.
\end{align*}
We point out that, when $h$ tends to zero, the area enclosed $A_{iii}(h)$ increases, 
while for $h$ close to one, $A_{iii}(h)$ achieves its minimum value. 

In a similar way, we can express the perimeter and the area enclosed by a set of type (iv) in terms of the 
mean curvature $h$. By denoting by 
\begin{equation}
\label{alphabeta}
\hat\beta:=\hat\beta(h)=\frac{\pi}{2}-\arcsin(h),\quad \hat\alpha:=\hat\alpha(h)=\arccos\bigg(\frac{\cos\hat\beta}{\lambda}\bigg),
\end{equation}
we have that necessarily $0<h\leq 1$, 
and that the perimeter and area enclosed by a set of type (iv) with curvature $h$ are equal to
\begin{align*}
P_{iv}(h)&=\frac{4\,\lambda\,\hat\alpha}{h} + \frac{4\,\arcsin(h)}{h}, 
\\
A_{iv}(h)=\frac{4\,\sin\hat\alpha}{h} + 2\,&\frac{\lambda\hat\alpha-\sin\hat\alpha\:\cos\hat\beta}{h^2\,} + 2\,\frac{\arcsin(h)-h\sqrt{1-h^2}}{h^2}. 
\end{align*}


Figure~\ref{fig:perfiles} shows several graphs where the perimeters of these types of sets are displayed, for 
different values of $\lambda$. 

\begin{figure}[htp]
\centering{
\subfigure[$\lambda=1.1$]{\label{conf1}\includegraphics[width=0.485\textwidth]{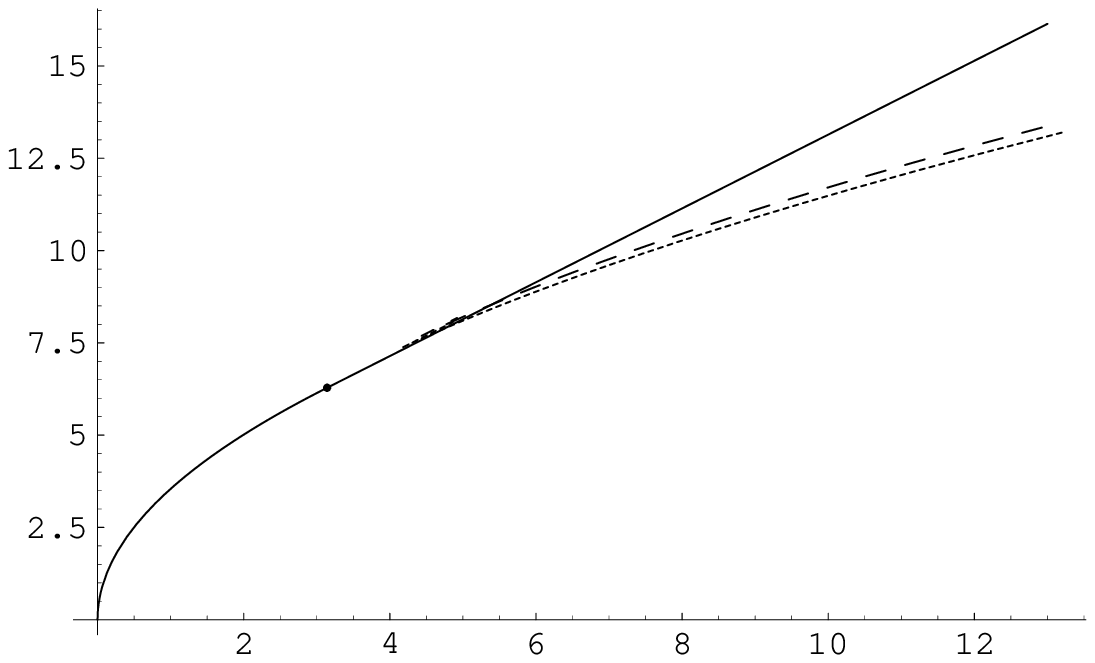}}
\hspace{0.01\textwidth}
\subfigure[$\lambda=2$]{\label{conf3}\includegraphics[width=0.485\textwidth]{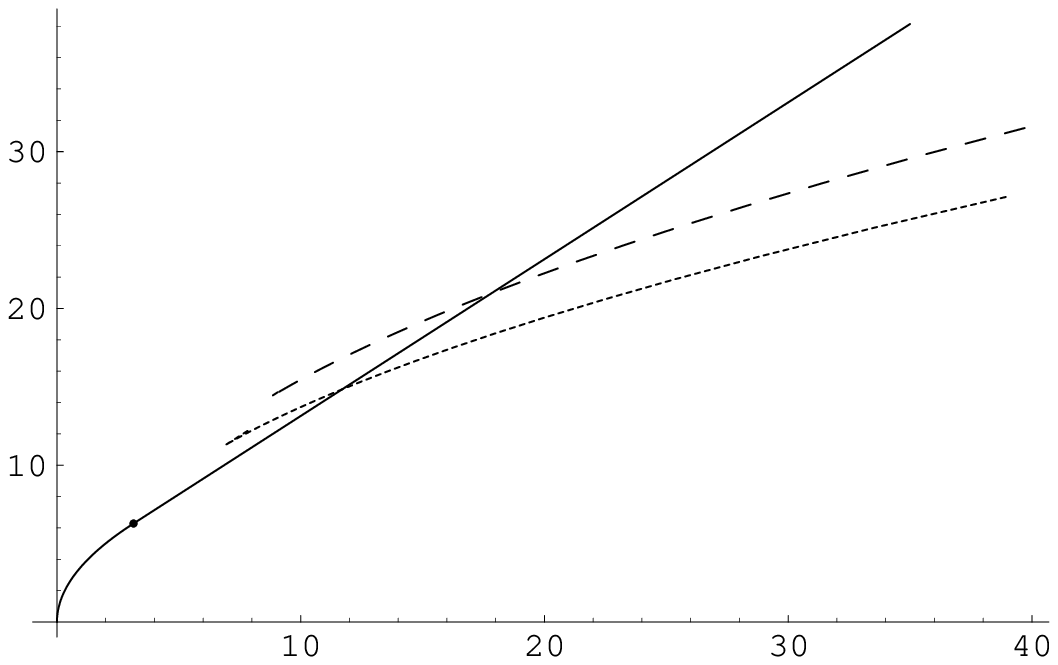}}
}
\\
\centering{
\subfigure[$\lambda=8$]{\label{conf6}\includegraphics[width=0.483\textwidth]{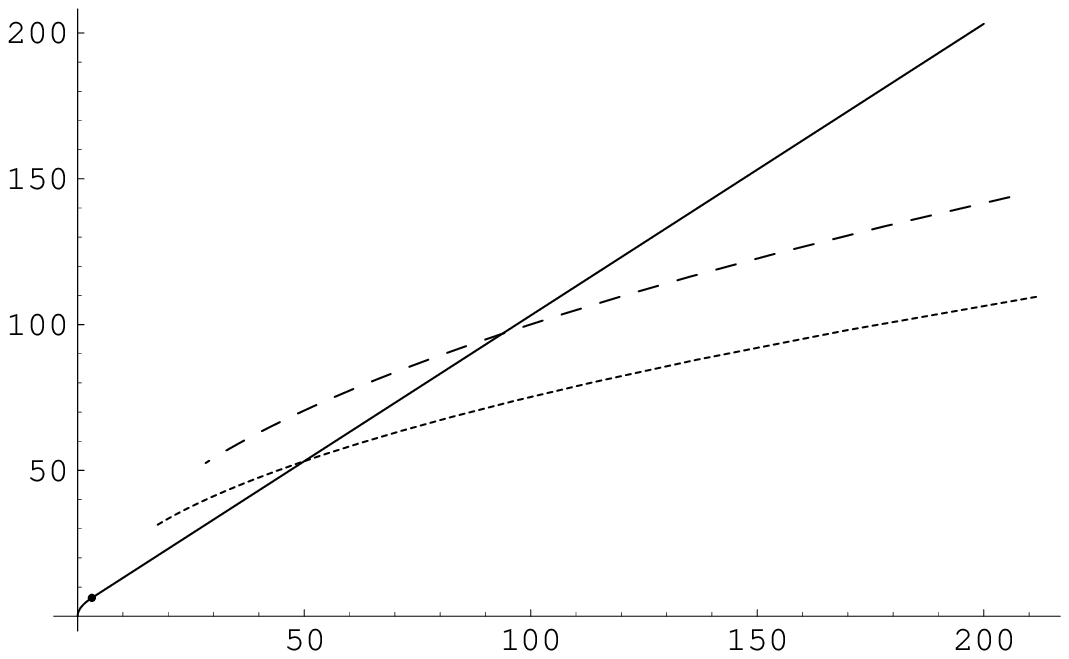}}
\hspace{0.01\textwidth}
\subfigure[$\lambda=1000$]{\label{conf7}\includegraphics[width=0.485\textwidth]{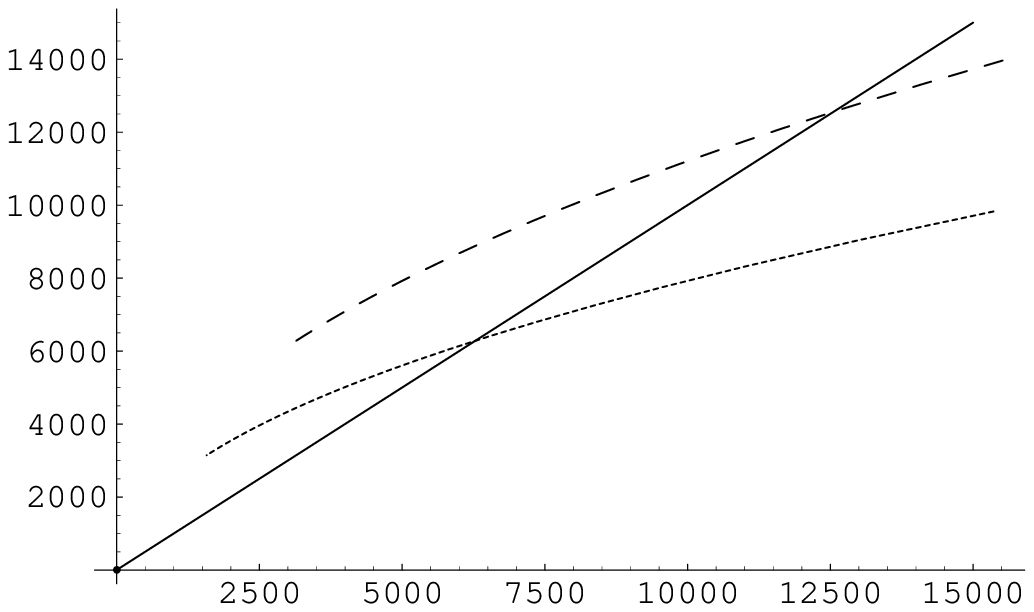}}
}
\caption{Comparison of perimeters in four different cases: the solid curve represents the transition between sets of 
type (i) and (ii), while the other dashed and dotted curves correspond to the perimeter of sets of type (iv) and (iii), 
respectively.}
\label{fig:perfiles}
\end{figure} 

For these and many others values of $\lambda$, we have obtained numerically the behaviour 
described in Conjecture~\ref{conj:strip}, as those graphs show. Moreover, it seems that
sets of type (iv) are always beaten by sets of type (iii), which are the 
solutions for large values of the area (see Theorem \ref{teostrip}). 
Thus, sets of type (iv) are not expected to be solutions. 

\begin{rmk} 
\label{re:re}
It is not difficult to justify that sets of type (ii) cannot be the 
solutions for large values of the area. Indeed, for any $\lambda>1$, 
and for curvature $h$ close enough to zero, we have $P_{iv}(h)<P_{ii}(A_f(h))$,  
which means that the set of type (iv) is better than the set of type (ii) for area $A_{iv}(h)$ 
(which shall be a large value of the area). 
In addition, the reverse inequality holds for values of $h$ close enough to one, and 
then sets of type (ii) are better for such (smaller) areas. 
We also have these same properties when considering type (iii) sets instead of type (iv) ones.
\end{rmk}

\begin{rmk}
Sets of type (iii) and (iv) possess an unexpected feature. They have curvature less than one but, for curvature slightly less, perimeter and area {\em decrease}. This can be noticed in Figure \ref{figreversing}, showing the graph of perimeter in terms of the enclosed area in case $\la=1.1$: the solid curve (the lower one) is associated to sets (iii), the dotted (upper one) to type (iv). Curvature is taken from 1 (corresponding to the ``left'' endpoint of each graph) to 0.8. Notice that initially (i.e. for curvature close to one) each graph ``goes left'', until a minimum (for both area and perimeter) is obtained in correspondence of the corner point of the graph.
\begin{figure}[htbp]
\begin{center} 
\includegraphics[width=11cm]{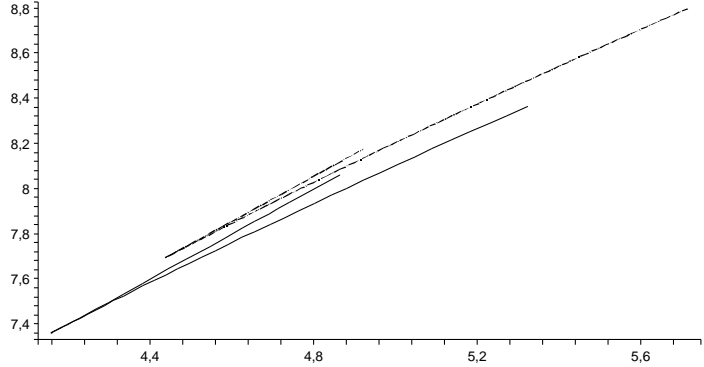}
\end{center}
\caption{When curvature goes from one to zero, the perimeter and area of sets (iii)-(iv) firstly decrease 
and then increase.}
\label{figreversing} 
\end{figure}
\end{rmk}

Even if without a complete result, in Theorem \ref{teostrip} we will characterise most of the isoperimetric sets for the strip density. First we have to consider the following ``line'' density setting, which we will see as a sort of ``blow-down'' of the strip density. Consider the density $g$ on $\R^2$ given by
\begin{equation}\label{linedensity}
g(x,y)=\left\{\begin{array}{ll}
\la & \text{if }y\neq 0\\
1 & \text{if }y=0\,.
\end{array}\right.
\end{equation}

\begin{theo}\label{theoline}
For the density in $\rr^2$ defined by $g$, isoperimetric sets exist for any given volume and are bounded 
by an arc meeting the horizontal axis with angle $\arccos \tfrac1\la$, and with the chord on the axis as in Figure \ref{figline}.
\begin{figure}[htbp]
\begin{center} 
\scalebox{1}{ 
\input{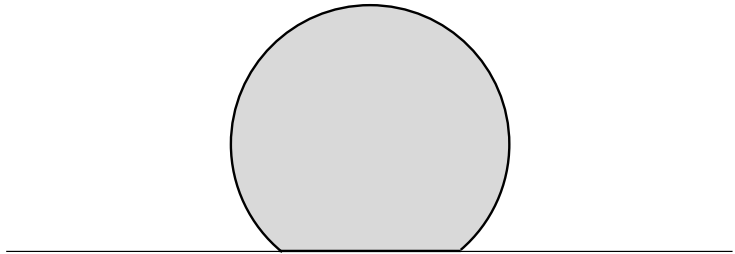} 
} 
\end{center}
\caption{Isoperimetric set for the line density.}
\label{figline} 
\end{figure}
\end{theo}
\begin{proof}
The existence of an isoperimetric set for any prescribed volume, as well as its connectedness, follows from a compactness argument similar to the one in Proposition \ref{PropExistN2}. As in Proposition \ref{WSteiner}, Schwarz symmetrisation with respect to the vertical axis ensures that an isoperimetric set symmetric with respect to such axis exists for any volume bound. We claim that an isoperimetric set with such a symmetry must be the one described in the statement; the fact that any isoperimetric set is symmetric with respect to the vertical axis easily follows by adapting one of the arguments of Proposition \ref{prop:simetrico} and Remarks \ref{rem:symm1} and \ref{rem:symm2}.

The regularity of the isoperimetric boundary ensures that for $\g_+:=\partial E\cap \{y>0\}$ one of the following holds:
\begin{itemize}
\item[($a_+$)] it consists of a complete circumference;
\item[($b_+$)] it consists of a single arc with chord a segment $\ell_+$ on the horizontal axis 
(possibly reducing to a single point);
\item[($c_+$)] it is empty.
\end{itemize}
The same possibilities, which we will call ($a_-$), ($b_-$) and ($c_-$), are given also for $\g_-:=\partial E\cap \{y<0\}$, with $\ell_-$ defined analogously. 
Taking into account the Snell law, our thesis is equivalent to proving that 
the isoperimetric set falls into one of the cases ($b_+$$c_-$) or ($c_+$$b_-$).

We can discard the cases ($a_+$$a_-$), ($a_+$$b_-$) and ($b_+$$a_-$): in fact, one could translate vertically the two components towards the origin until an ``irregular contact'' is obtained.

The case ($b_+$$b_-$) can be discarded too. First of all, the two segments $\ell_+$ and $\ell_-$ must coincide; otherwise (by vertical symmetry) one would be contained in the other, say $\ell_-\subsetneq\ell_+$. The arcs $\g_+$ and $\g_-$ meet the horizontal axis satisfying Snell law, and so one of them forms (with the same axis) an angle $\al:=\arccos\tfrac1\la$, and the other one will form an angle $\pi-\al$. This easily implies that the chords $\ell_+$ and $\ell_-$ have the same length, and so that our set is a round ball. In particular, perimeter and area are the same as the cases ($a_+$$c_-$) and ($a_-$$c_+$), which we are going to discard in a while.

The case ($c_+$$c_-$) is of no concern, as it corresponds to null volume. For sets $E$ falling into cases ($a_+$$c_-$) or ($a_-$$c_+$), i.e. when $E$ is a ball, the ratio $P_g(E)^2/A_g(E)$ is equal to $4\pi\la$. On the other hand, when $E$ falls into ($b_+$$c_-$) (the case ($c_+$$b_-$) is analogous) it consists of a circle minus the area enclosed by the chord $\ell_+$, the angle at the center determined by $\ell_+$ being equal to $2\al$ by Snell law. Easy computations yield that
\begin{align*}
\frac{P_g(E)^2}{A_g(E)}=\frac{4((\pi-\al)\la+\sin\al)^2}{\la(\pi-\al+\sin\al\cos\al)}&=\frac{4\la^2((\pi-\al)+\sin\al\cos\al)^2}{\la(\pi-\al+\sin\al\cos\al)} 
\\
&< 4\la(\pi-\al+\al)=4\pi\lambda, 
\end{align*}
where we have used that $0<\al=\arccos\tfrac1\la<\pi/2$. This concludes the proof.
\end{proof}

We also need to introduce the function $\arc(x)$ defined as the perimeter of a circular arc with unit chord and area $x$. Let us briefly recall some of its properties (see \cite[Section 15.5]{Mor08Geo}):
\begin{itemize}
\item $\arc$ is convex in $[0,\pi/8]$ and concave in $[\pi/8,+\infty[$;
\item $\arc(\pi/8)=\pi/2$ and $\arc'(\pi/8)=2$.
\end{itemize}
We stress that the boundary case $x=\pi/8$ corresponds to a semicircle of radius $1/2$. \vspace{3mm}

We now collect in the following theorem the partial results we are able to prove for this strip density.
\begin{theo}\label{teostrip} 
The isoperimetric sets for area $v>0$ for the strip density, defined by $1$ in the strip $\{|y|\leq1\}$ and $\la>1$ outside, are
\begin{itemize}
\item[(a)] balls of type (i) if $v\leq\pi$;
\item[(b)] sets of type (ii) if $\pi<v\leq v_0$;
\item[(c)] sets of type (iii) or (iv) if $v_0\leq v\leq v_1$;
\item[(d)] sets of type (iii) if $v\geq v_1$,
\end{itemize}
for suitable $v_1\geq v_0>\pi$. In particular, for area $v_0$ there are at least two different isoperimetric sets.

\noindent
Moreover, type (iv) never occurs for $\la\geq 4/\pi$, i.e. in this case one can take $v_0=v_1$.
\end{theo}
\begin{proof}
Set $E$ an isoperimetric region for area $v$. 
To verify part (a) of the statement, we just need to show that any set 
of type (ii), (iii) or (iv) has area greater than $\pi$. This is clear for sets of type (ii).

For sets of type (iii) we observe that, if the three arcs have radius $R>2$, 
then the arc outside the strip is greater than the semicircle with same radius, 
and so the enclosed area is greater than the area of that semicircle (times the density $\la$), 
whence greater than $\la\pi R^2/2>\pi$. 
Otherwise we must have $1<R\leq 2$ and it is easy to see that 
the internal arcs meet $\{y=1\}$ in such a way that 
the angle between the segment $\{y=1\}\cap E$ 
and each arc inside the strip is greater than $\pi/2$. 
Using the Snell law it is not difficult to observe that 
the sum of the central angles associated with the three arcs is greater than $2\pi$; 
moreover, the corresponding sectors of circle do not overlap and are contained in $E$. 
This means that the area of $E$ is greater than that of a circle of radius $R$, 
hence greater than $\pi$.

For sets of type (iv) we can use a similar argument 
and obtain that the four central angles 
have sum greater than $2\pi$. 
Let us check that the four associated sectors do not overlap; 
in this way the area of the candidate is greater than 
that of the circle with same radius $R>1$ and are done.
It is sufficient that the sectors corresponding 
to the upper and lower arcs do not overlap, i.e., 
that the center of the upper arc has positive $y$-coordinate 
(and so the center of the lower arc has negative $y$-coordinate). 
Since the parameter $\hat\al$ in \eqref{alphabeta} 
denotes half of the central angle corresponding to the upper arc, 
the $y$-coordinate of its center equals $1-R\cos\hat\al=1-\tfrac 1\la>0$.

%
We claim that the graphs representing the perimeter in terms of the area enclosed, 
corresponding to sets of type (ii) and (iii), intersect only once (see the graphs of Figure \ref{fig:perfiles}).
Indeed, the curvature of the sets of type (iii) is always less than one 
(see \eqref{alphabetabis}), which is the curvature of any set of type (ii). 
Since curvature coincides with the \emph{slope} of those graphs, 
this implies 
that those graphs 
intersect at most once. 
Moreover, type (iii) does not exist for area $\pi$ and, by Remark~\ref{re:re}, 
beats type (ii) for large area. 
%
%
Therefore, there exists some $a_0>\pi$ such that sets of type (ii) are better than sets of type (iii) 
for areas less than $a_0$, while sets of type (iii) are better
for areas greater than $a_0$. 
The same reasoning can be done for sets of type (iv): sets of type (ii) are better than type (iv) ones
for areas less than $a_1$, the contrary for greater areas. Part (b) thus follows by noticing that any type-(i) set has area at most $\pi$; the value $v_0$ can be characterised by $v_0:=\min\{a_0,a_1\}>\pi$.



Statements (c) and (d) require more work. By contradiction, suppose there is a sequence of areas $v_j\to\infty$ such that the related isoperimetric sets $E_j$ are of type (iv). Let $R_j$ be the radius of the curvilinear components of $\partial E_j$; notice that each of the components in $\{y>1\}$ and in $\{y<-1\}$ meets the strip with an angle lying in $[\arccos\tfrac 1\la,\pi-\arccos\tfrac1\la]$. This forces each of these two arcs to bound a region of area at least $cR_j^2$ ($c>0$ depending only on $\la$), while the area of $E_j$ contained in the strip is of order $R_j$. In particular, for large $j$ a ``big'' portion (at least $Cv_j$, $C=C(\la)>0$) of the area of $E_j$ lies in the half-plane above the strip, and another ``big'' portion lies below the strip.

Let us perform the following ``blow-down'' operation: for any fixed $j$ we make a dilation of a factor $v_j^{-1/2}$ so that $E_j$ is mapped into a set of area 1. More precisely, we consider the ``thin strip'' density $f_j$ on $\R^2$ given by
$$
f_j(x,y)=\left\{\begin{array}{ll}
\la & \text{if }|y|>\tfrac{1}{\sqrt{v_j}}\\
1 & \text{otherwise}\,.
\end{array}\right.
$$
Clarly $f_j\to g$ pointwise, where $g$ is the line density defined in \eqref{linedensity}, and the sets $F_j:=E_j/\sqrt{v_j}$ are isoperimetric and of unit weighted area. Being connected they satisfy the diameter estimate \eqref{estdiamper}. Up to horizontal translations, we can therefore suppose that they are contained in a big ball $B$, and a compactness argument ensures that a subsequence converges in $L^1$ to a certain set $F$. This means that
$$
A_{f_j}(F_j)\to A_g(F)=1,
$$
where $g$ is the function defined by \eqref{linedensity};  
this is because $g$ and $f_j$ differ only on the thin strip, whose intersection with the ball $B$ has area going to 0. Semicontinuity of the perimeter and the inequality $f_j\geq g$ imply
$$
P_g(F)\leq\liminf_{j\to\infty} P_g(F_j)\leq\liminf_{j\to\infty} P_{f_j}(F_j).
$$
It follows that $F$ has to be isoperimetric for the density $g$; in fact, were another unit area set $F'$ isoperimetrically better than $F$ we would obtain
$$
\frac{P_g(F')^2}{A_g(F')}+\ep < \frac{P_g(F)^2}{A_g(F)}\leq \liminf_{j\to\infty}\frac{P_{f_j}(F_j)^2}{A_{f_j}(F_j)}
$$
for some positive $\ep$, and so $F_j$ could not be isoperimetric (for $f_j$) for large $j$.

We will reach a contradiction if we show that $F$ cannot be one of the isoperimetric sets described in Theorem \ref{theoline}. In fact, for large $j$ the part of $F_j$ lying above the ``thin strip'' has area at least $C>0$, and the same holds for the part below it. This must happen also for $F$, and so it cannot coincide with one of the sets in Theorem \ref{theoline}.

It remains to show that for $\la\geq4/\pi$, case (iv) can never occur. Consider in fact any such set; let $L$ be the length of the chords of the upper and bottom caps and $A$ be the (Euclidean) area enclosed by each of them. Replace these two caps with a segment of length $L$ on $\{y=-1\}$ and a cap in $\{y>1\}$, enclosing Euclidean area $2A$ and lying upon the same chord as the removed upper cap. To prove that we have decreased perimeter we have to show that
$$
\la L\, \arc(2A/L^2) + L <2\la L\, \arc(A/L^2)\,.
$$
Setting $x:=A/L^2$, this is equivalent to show that
$$
h(x):=2\arc(x)-\arc(2x)>1/\la.
$$
Since $h'(x)=2\arc'(x)-2\arc'(2x)$, the convexity/concavity properties of $\arc$ ensure that $h'(x)<0$ for $x\leq \pi/16$ and $h'(x)>0$ for $x\geq \pi/8$. Therefore the minimum of $h$ is attained at some $x_{\rm min}\in]\pi/16,\pi/8[$. Since $\arc$ is convex in $[0,\pi/8]$ we have
$$
\arc(x_{\rm min})>\arc(\pi/8)+\arc'(\pi/8)(x_{\rm min}-\pi/8)=\pi/4+2x_{\rm min}
$$
while concavity in $[\pi/8,+\infty[$ yields
$$
\arc(2x_{\rm min})<\arc(\pi/8)+\arc'(\pi/8)(2x_{\rm min}-\pi/8)=\pi/4+4x_{\rm min}\,.
$$
The last two equations imply that $h(x_{\rm min})>\pi/4\geq 1/\la$ as desired.
\end{proof}

\begin{rmk}
The bound $4/\pi$ in the last assertion of Theorem \ref{teostrip} is not optimal. In fact, a closer examination
of the inequalities we have used leads to $h(x_{\rm min})=k>\pi/4$, and actually the statement holds for $\la> 1/k$. Unfortunately, since $h(0)=1$ one has $k< 1$: in other words, our argument of ``substituting upper and bottom caps with a unique one'' fails for small values of $\la$.
\end{rmk}

\subsection{The ball}\label{SecBalls}
In this section we denote by $f$ the ``ball'' density on $\R^2$ taking values
$\la\in\R\setminus\{1\}$ on the open ball $B=B(0,1)$ and 1 on $\R^2\setminus \overline B$. The value of $f$ on $\partial B$ is $\min\{1,\la\}$ in order to guarantee its lower--semicontinuity. We begin with an existence result for the isoperimetric problem.

\begin{theo}
For the ball density, isoperimetric sets exist for any given area.
\end{theo}
\begin{proof}
Let $v>0$ be the area bound; it will be sufficient to prove that there exists $R>0$ and a minimising sequence $\{E_j\}$ such that
$$
|E_j|_f=v,\qquad  P_f(E_j)\to\inf\{P_f(F):|F|_f=v\}\qquad\text{and}\qquad E_j\subset B(0,R)\,.
$$
In this way, we will reduce to an isoperimetric problem in a space with finite total mass, 
where existence of isoperimetric sets is a well--known result (see \cite{Mor08Geo}).

To do this, consider any minimising sequence $\{F_j\}$ with $|F_j|_f=v$; it is not restrictive to assume $F_j$ to be smooth open sets, and so one can write $F_j=\cup_i F_{j,i}$, where $F_{j,i}$ denote the connected components of $F_j$. If $i$ is such that $F_{j,i}\cap B\neq\emptyset$, then estimate \eqref{estdiamper} easily provides $R_1>0$ (depending only on $v$) such that $F_{j,i}\subset B(0,R_1)$. For the other components, we can substitute $\cup_i\{F_{j,i}:F_{j,i}\cap B=\emptyset\}$ with a unique ball $B(x,R_j)$ (for certain $x\in\R^2$) enclosing the same volume,
thus defining the set $E_j:=B(x,R_j)\cup\bigcup\{F_{j,i}:F_{j,i}\cap B\neq\emptyset\}$. Notice that $R_j$ is bounded uniformly in $j$ by some $R_2$ depending only on $v$, since $|E_j|_f=v$, and that $P_f(E_j)\leq P_f(F_j)$. Anyway, we have a lot of freedom in the choice of the center $x$; i.e., we can move $B(x,R_j)$ towards the origin, without touching the other components of $E_j$, until we obtain (say) $E_j\subset B(0,R_1+2R_2+1)$.
\end{proof}

The nature of the isoperimetric problem, as we will see in a moment, is completely different in the two cases $\la>1$ and $\la<1$. Therefore, we treat them separately.

\subsubsection{The ball with density $\la>1$.}
\label{ballmayor}

We start with a description of all possible candidates for isoperimetric set. The first one is
a round ball, and, as we shall see in Theorem \ref{ballmagg}, balls are isoperimetric for volumes $v>\lambda \pi$. 
Other more complicated candidates arise when we look for
isoperimetric sets with area $v$ smaller than $\la\pi$. 

\begin{prop}\label{ballmino}
If the prescribed area $v$ is smaller than $\la\pi$, then an isoperimetric
set is one of the following (see Figure~\ref{ex_isop}): 
\begin{itemize}
\item[(a)] a round ball contained in $\R^2\setminus B$;
\item[(b)] the region enclosed by an arc of $\partial B$ and another internal circular arc (inside $B$), the angle between them being equal to $\arccos \tfrac{1}{\la}$
  (i.e., the limit case in the Snell law). 
\end{itemize}
\end{prop}
\begin{proof}
In what follows, we will denote by $E$ an isoperimetric set of area
$v$ and by $\g=\partial E$ its boundary.

{\em Step 1: $E$ is either a ball outside $B$ (i.e. a set of type $(a)$) or the closure of each of its components
is simply connected and intersects both $B$ and its complement.} 

It is not difficult to see that the isoperimetric boundary $\g$ can be decomposed into countably many rectifiable Jordan curves $\gamma_j$, and that $E$ coincides with the union of the connected and simply connected open regions $E_j$ enclosed by each $\g_j$. Otherwise, we could write $E_j=F(\g_j)\setminus G_j$, where $F(\g_j)$ is the region enclosed by a certain Jordan curve $\g_j$ and $G_j$ is a proper subset of $F(\g_j)$. Since our density is rotationally invariant, we could rotate $G_j$ about the origin until it touches $\g_j$: perimeter and area are unaffected, but regularity fails. 

Since balls outside $B$ are better than inside, there is no $E_j$ with $\overline{E_j}\subset B$, or we could substitute all such regions with a unique ball outside $B$, with same area and less perimeter. Similarly, there is at most one single component outside $B$, which must be a ball. If this is the case, there is no other component $E_j$ such that $\overline{E_j}$ intersects both $B$ and its exterior, otherwise we could move our external ball until it touches $E_j$, thus violating regularity. Therefore, the isoperimetric set would be a ball like in part $(a)$ of the statement.

In the rest of the proof we will focus our attention on the remaining case that any $\overline{E_j}$ intersects both $B$ and its exterior. 

{\em Step 2: Only a finite number of internal arcs is allowed: as a consequence, $E$ has regular trace inside $B$.}

Recall that $\g\cap B$ and $\g\setminus \overline B$ consists of circular arcs $c_j$ with endpoints on $\partial B$ (we have already excluded whole circles),
 with the same constant curvature $\kappa$. Moreover, $E$ must contain all the regions $F_j$ enclosed by a $c_j$ and the corresponding arc on $\partial B$, 
determined by the endpoints of $c_j$ (notice that, for internal arcs, there are two possible regions of this type, a bigger and a smaller one). Otherwise, we 
could use a cutting technique. 

Suppose by contradiction that there are internal arcs $c_j$ as in Figure \ref{internalarcs} of arbitrarily small length and set $\al_j$ to be the associated central angle. 
Due to curvature $\kappa$, the area of the region $F_j$ enclosed by $c_j$ is of order $\al_j^3$. By cutting $F_j$ from our set we would 
gain perimeter $\sim (\la-1)\al_j$, and we could restore the lost area $\sim \al_j^3$ with some ball of radius (and perimeter) $\sim\al_j^{3/2}$. This is a contradiction, since for sufficiently small $\al_j$ we would overall gain perimeter.

\begin{figure}[h!tbp]
\begin{center} 
\scalebox{1}{ 
\input{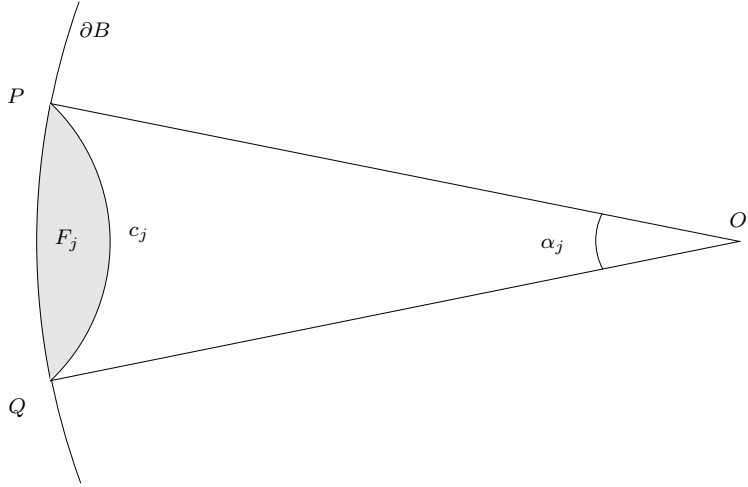} 
} 
\end{center}
\caption{Internal arcs enclose a fixed area amount.}
\label{internalarcs} 
\end{figure}

As a straightforward consequence, only a finite number of internal $c_j$'s is allowed. This implies also that the number of connected components $E_j$ is finite: 
in fact, $\overline E_j\cap B\neq \emptyset$, whence the boundary $\g_j$ must meet the open ball $B$, and so contains an internal arc $c_j$.

{\em Step 3: External arcs cannot meet $\partial B$ with angle strictly greater than $\pi/2$.}

This is a technical point: it will allow us to use arguments like ``put external arcs in succession, so that two consecutive ones have a common endpoint'',
which would be precluded if one angle were greater than $\pi/2$. In fact, by doing such an operation it could happen that two consecutive regions overlap; it 
is evident that this pathology cannot occur if the meeting angles are less than $\pi/2$. 

Suppose now that there exists such an external arc $c$, with endpoints $P$ and $Q$; let us prove that the distance between $c$ and any other external arc is 
bounded away from 0. Otherwise, a sequence of external arcs $c_k$ would accumulate at $P$ (or $Q$), so some arc near $P$ would lie ``below'' $c$ (i.e., there exists a half--line from the origin meeting both $c$ and $c_k$); say that the 
accumulating sequence of arcs stays ``on the left'' of $P$. Since there is only a finite number of internal arcs we just have the following two possibilities:
\begin{itemize}
\item there is a small arc $a\subset\partial B$, with $P$ as the right endpoint (i.e. $a$ is staying ``on the left'' of $P$) which is contained in one of the 
arcs on $\partial B$ determined by some internal arc $c_j$. Informally, we could say that the sequence $c_k$ ``rests'' on the arc $a$. We could then use a small segment $s$ to join an external arc $c_k$ ``below'' $P$ to some point on $c$ close to $P$, and add the region enclosed by $s
,\ c,\ a$ and the $c_k$'s. In this way we have gained both perimeter and area, and we could use cutting arguments to restore area and gain perimeter again, which is contradictory.
\item there is a small arc $a$ on $\partial B$, with $P$ as the right endpoint, 
which is not contained in any of 
the arcs on $\partial B$ determined by the internal arcs $c_j$. This implies that each external arc $c_k$, together with the arc it determines on $\partial B$, bounds a connected component contained in the region with density 1, which could be cut away and replaced by balls to restore area and gain perimeter.

\end{itemize}

Also, it is not possible for $P$ and $Q$ to lie in the {\em interior} of the arc $a$ of $\partial B$ determined by some internal arc. Otherwise we could use 
a segment $s$ to join some point on $c$, close to $P$, to some point on $a$ ``below'' $P$, and add the region enclosed by $c$, $a$ and $s$, thus gaining both
 perimeter and area.

The previous arguments imply that the connected component $E_j$, whose boundary contains $c$, is ``free'' to move left or right: i.e., 
unless $E_j$ is the only component, we can rotate it about the origin until one of the following happens
\begin{itemize}
\item $c$ touches another external arc at a point outside $\overline B$;
\item some endpoint of another external arc becomes too close to $P$ or $Q$;
\item some internal arc of $\g_j$ touches another internal arc at a point in $B$;
\item some internal arc not of $\g_j$ becomes too close to $P$ or $Q$.
\end{itemize}
Each of these possibilities leads to a contradiction (since in any case we could use cutting arguments) unless there exists only one connected component $E$.

If this is the case, consider an internal arc $c_j$: it is not difficult to realise that the arc on $\partial B$ determined by $c_j$'s endpoints
must be contained in the arc
$PQ$ (again, the one ``internal'' to the region enclosed by $c$). 
Otherwise, it would be contained in the opposite arc $PQ$ 
(we have already seen that the endpoints of $c_j$ cannot lie one 
on the arc $PQ$ and one on the other arc $PQ$, or we could cut at $P$ or $Q$), 
violating connectedness. Since any internal 
arc ``rests'' on $PQ$ and there is only one connected component, 
there cannot exist another external arc apart from $c$.

Therefore, the boundary of $E$ is made by $c$ plus a finite number of internal arcs $c_j$, any of them ``resting'' on the arc $PQ$. Were there more than one,
 we could rotate them until they touch, contradicting minimality also if the touching point is on $\partial B$ 
(the Snell law (\ref{eqfresnel}) would not hold). Therefore just an
 internal arc $c'$ is allowed: by rotating it, we can suppose it to have an endpoint in $P$, where the Snell law must hold. By symmetry, the other endpoint meets $\partial B$ with the same angle, and so must coincide with $Q$. 

Recall that $c$ and $c'$ have the same curvature: so, they can be either symmetric with respect to the segment
$PQ$ or form a complete circle. In both cases, the Snell law would be violated unless the meeting angle is exactly $\pi/2$ yielding a contradiction.

{\em Step 4: Only a finite number of external arcs is allowed: in particular $E$ has regular traces on $\partial B$.}

By contradiction: suppose that outer arcs accumulate at a point $P\in\partial B$; we can then find an arc $\ell\subset\partial B$ such that
\begin{itemize}
\item $P$ is an endpoint of $\ell$;
\item there are infinitely many external arcs $C_j$ whose endpoints lie on $\ell$;
\item the trace of $B\cap E$ on $\partial B$ is constantly 0 or 1 on $\mathring \ell$ (in other words: the finitely many arcs $\ell_j$ on $\partial B$, determined by internal arcs $c_j$, are such that $\mathring\ell\subset \ell_j$ or $\mathring\ell\cap\ell_j=\emptyset$).
\end{itemize}
We can therefore re-order the $C_j$'s (and the corresponding enclosed regions $F_j$) in sequence, in such a way that they still ``rest'' on $\ell$ and any two consecutive $C_j$ have a common endpoint on $\partial B$. In this way, area is clearly not affected, and so does the perimeter, since the third assumption on $\ell$ ensures that we are not creating (nor destroying) perimeter on $\partial B$. Regularity clearly fails with this construction, since the number of internal arcs is finite. 

{\em Step 5: The isoperimetric set is connected, enclosed by a piecewise $\ci^1$ Jordan curve, and the Snell law \eqref{eqfresnel} holds.}

As we have seen, the boundary $\g\setminus\partial B$ is made by a finite number of arcs; therefore, the whole $\g$ is made by a finite number of arcs which can be internal, external or subarcs of $\partial B$. In what follows, we will call ``external'' also an arc of $\g$ constituted by a subarc of $\partial B$. The regularity of traces of $E$ on $\partial B$ implies that the Snell law \eqref{eqfresnel} must hold at any endpoint on $\partial B$ of any arc of $\g$. More precisely, at any of these endpoints a unique external and a unique internal arc meet, with the Snell law satisfied. This holds true also when the external arc is the limit case, i.e. a subarc of $\partial B$.

Also, there is only a finite number of connected components (which, as we have already seen, are simply connected too). Suppose there
were more than one: we could rotate one of them about the origin until meeting another, contradicting the Snell law. 

{\em Step 6: An external arc cannot meet $\partial B$ transversally.}

%
Suppose by contradiction that this
happens at a point $P\in\partial B$; then (part of) $\g$ consists of two circular arcs $c_1$ (outside $B$) and $c_2$ (in $B$), of the same
radius, which intersect at $P$ respecting the Snell law \eqref{eqfresnel}: see Figure \ref{ball2}. Since
$\la>1$, the angle $\al_1$ at $P$ between $c_1$ and $\partial B$ is
strictly smaller than the angle $\al_2$ at $P$ between $c_2$ and
$\partial B$, provided $\al_1\neq \pi/2$. In some sense, the
discontinuity of the density causes a rotation, of center $P$ and of
an angle $\al_2-\al_1>0$, of the circle $c_1$. We recall
that the curvatures of $c_1$ and $c_2$ (with respect to the inner normal), and not just their radii, coincide. 

\begin{figure}[h!tbp]
\begin{center} 
\scalebox{1}{ 
\input{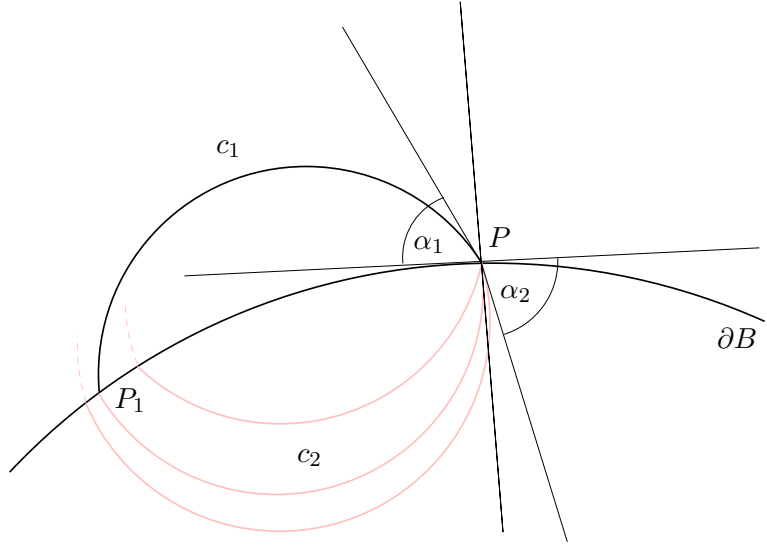} 
} 
\end{center}
\caption{Possible prolongations of $c_2$.}
\label{ball2} 
\end{figure}

Let then $P_1,P_2$ be the other points where $\partial B$ meets,
respectively, $c_1$ and $c_2$, and suppose $P_1\neq P_2$: by
prolonging $\g$ after $P_2$ (thus obtaining an arc $c_3$, see Figure \ref{ball3}) we will fall 
into a contradiction. In fact, the Snell law
holds again, and so the angle between $c_3$ and $\partial B$ is
exactly $\al_1$: this means that $c_3$ can be obtained from $c_1$
by means of the rotation centered at the origin which moves $P_1$
to $P_2$. The arcs $c_1$ and $c_3$ do not intersect (this happens,
e.g., when $P_2$ lies in the arc $P_1P$), or this would lead to a
contradiction. If $P_2$ does not lie on the arc $P_1P$ and $c_1,c_3$
do not intersect, then $c_3$ meets $\partial B$ at another point $P_3$
on $P_2P_1$. Prolong then $c_3$ from $P_3$ to obtain $c_4$: this arc
will surely meet $c_2$, since it can be obtained by rotation and
$P_3\in P_2P_1$. This is a contradiction. 

\begin{figure}[h!tbp]
\begin{center} 
\scalebox{1}{ 
\input{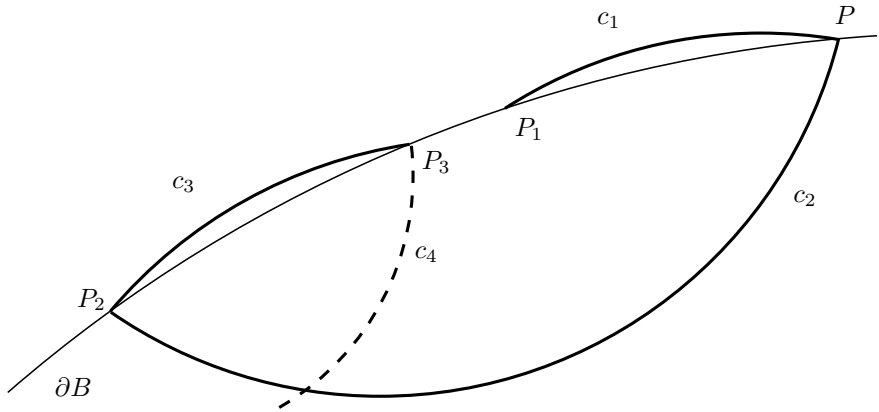} 
} 
\end{center}
\caption{Self intersections cannot happen.}
\label{ball3} 
\end{figure}

If $P_1=P_2$, then $c_1,c_2$ (circles with same curvature) are either symmetric with respect to the segment
$PP_1$ or form a complete circle, and this contradicts the Snell law unless
$\al_1=\al_2=\pi/2$, i.e. $c_1$ and $c_2$ form a complete circumference $c$ (Figure
\ref{ball4}). Let $Q$ be the center of $c$, $r$ its radius, and
set $\be:=\widehat{P_1QP}$; noticing that the radii $QP$ and $QP_1$ are
tangent to $\partial B$, we obtain that the density area $v_f(c)$ of the circle
is less than 
$$
r^2[\pi+(\la-1)\cdot(\text{Area of sector }P_1QP)]=r^2[\pi+(\la-1)\be/2]\,.
$$
Since the density perimeter of the circle $c$ is
$r(2\pi+(\la-1)\beta)$, we obtain that 
$$
\frac{{P_f(c)}^2}{v_f(c)}>
\frac{r^2(2\pi+(\la-1)\beta)^2}{r^2[\pi+(\la-1)\be/2]}=
4[\pi+(\la-1)\be/2]>4\pi\,, 
$$
which is the isoperimetric ratio of any ball outside $B$. It follows
that $c$ cannot be isoperimetric. 

\begin{figure}[h!tbp]
\begin{center} 
\scalebox{1}{ 
\input{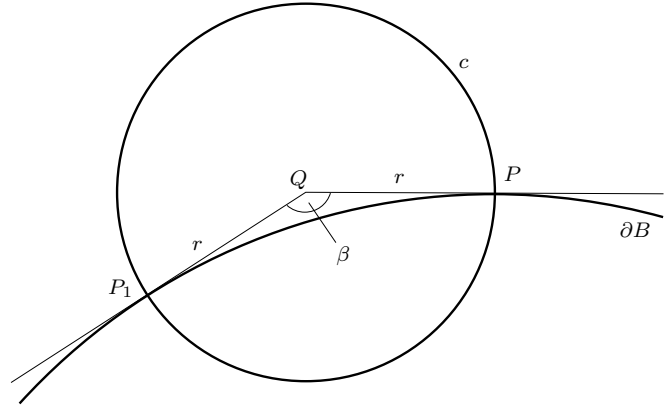} 
} 
\end{center}
\caption{A particular ball.}
\label{ball4} 
\end{figure}

{\em Step 7: Reduction to case $(b)$.}

In Step 6 we excluded external arcs which are not subarcs of $\partial B$ and have two endpoints on $\partial B$. In fact they would be tangential to $\partial B$, and then would either enclose the whole $B$ (contradicting assumption $v<\la\pi$) or a ball external to $B$.
%

%

It follows that the isoperimetric boundary $\g$ is composed by a finite number of internal $c_j$'s and some subarcs of $\partial B$. Were the internal $c_j$'s more than one, we could rotate one of them until it touches another one, thus violating minimality at the touching point. Therefore, the isoperimetric set is the one described in part $(b)$ of the statement.
\end{proof}


\begin{figure}[h!tbp]
\begin{center} 
\scalebox{1}{ 
\input{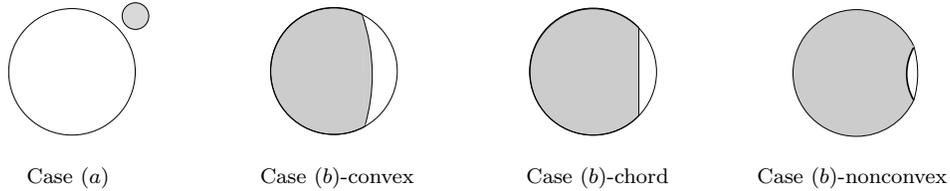} 
} 
\end{center}
\caption{Isoperimetric sets, $v<\la\pi$.}
\label{ex_isop} 
\end{figure}

We can now conclude this section with the main result.

\begin{theo}\label{ballmagg}
For the plane with density $\lambda>1$ on the unit ball and $1$ elsewhere, 
there exists a $0<v_0<\lambda\pi$ such that
\begin{enumerate}
 \item 
if $v\leq v_0$, the isoperimetric set with area $v$ is a ball contained in $\rr^2\setminus B$ (type \ref{ballmino}(a));
\item
if $v_0\leq v\leq \lambda\pi$ the isoperimetric set is of type \ref{ballmino}(b);
\item
if $v\geq \lambda\pi$, the isoperimetric set is a round ball containing 
the ball $B$.
\end{enumerate}
All three subcases of type \ref{ballmino}(b) of Figure~\ref{ex_isop}, convex, chord and nonconvex, actually occur for some values of $\lambda$.
\end{theo}
\begin{proof}
{\em Step 1: If $v\geq \lambda\pi$, the isoperimetric set is a round ball containing $B$.} 

\noindent
Let $E$ be any set of weighted area $v$: since the region with density
$\la$ has Euclidean area $\pi$, the Euclidean area of $E$ is at least
$v-(\la-1)\pi$ and by classical isoperimetry its Euclidean perimeter
is at least that of the ball with the same area, i.e. the one with
radius $(\tfrac{v}{\pi}+1-\la)^{1/2}$. Since $f\geq 1$, this also gives a
lower bound for the $f$-perimeter of $E$; however, any ball as in
the statement is such that these inequalities are in fact equalities, and
so is isoperimetric. Conversely, for any other set, 
at least one of the inequalities is strict, and so it cannot be isoperimetric.

{\em Step 2: If $v$ is small, the isoperimetric set is a ball of type $(a)$.} 

\noindent
Suppose by contradiction this is not true: 
then the isoperimetric set $E$ is a $(b)$-convex type (see Remark~\ref{evolution}), 
determined by an arc $PQ$ on $\partial B$ and another one of center
$O'$, meeting the arc $PQ$ with angle $\be_0=\arccos(1/\la)$:
as in Figure \ref{smallv}.

\begin{figure}[h!tbp]
\begin{center} 
\scalebox{1}{ 
\input{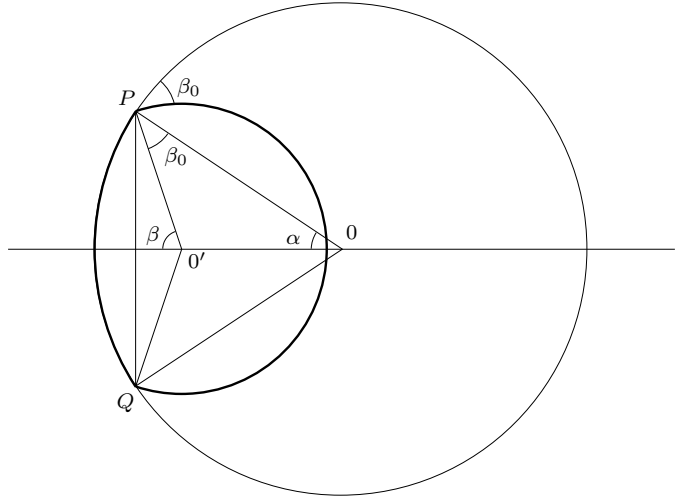} 
} 
\end{center}
\caption{Area $v\downarrow 0$.}
\label{smallv} 
\end{figure}

Let $\al$ be the angle
$\widehat{POO'}=\widehat{POQ}/2$ and $\be=\widehat{PO'Q}/2$; we have
$\be=\al+\be_0$ (since $\widehat{O'PO}=\be_0$) and it is easy to compute the density perimeter of $E$
$$ 
P_f(E)=P_f(E_\al)=2\left(\al+\la\frac{\sin\al}{\sin\be}(\pi-\be)\right)
$$
and the density area of $E$
$$
|E|_f=|E_\al|_f=\la\left[\al-\sin\al\cos\al
  +\left(\frac{\sin\al}{\sin\be}\right)^2(\pi-\be+\sin\be\cos\be)\right].  
$$
We will get a contradiction if we show that $\lim_{\al\to
  0}P_f(E_\al)^2/|E_\al|_f>4\pi$. To do this, we express $\sin\be,\cos\be$ and $\la$ as functions of $\al$ and $\be_0$ (remember that
$\be=\al+\be_0$ and $\la=1/\cos(\be_0)$) and, through a straightforward Taylor expansion, we
achieve 
$$
P_f(E_\al)= 2\al\left( 1+\frac{\pi-\be_0}{\sin\be_0\cos\be_0}\right)+O(\al^2)\qquad\text{and}\qquad |E_\al|_f= \frac{\al^2}{\sin\be_0}\left( 1+\frac{\pi-\be_0}{\sin\be_0\cos\be_0}\right)+O(\al^3)\,.
$$
Then
$$
\lim_{\al\to
  0}\frac{P_f(E_\al)^2}{|E_\al|_f}=
4\sin\be_0\left(1+\frac{\pi-\be_0}{\sin\be_0\cos\be_0}\right) =  4\sqrt{1-t^2}\left( 1+ \frac{\pi-\arccos t}{t\sqrt{1-t^2}}\right)\,,
$$
where we put $t:=1/\la\in(0,1)$. We have to prove that
$$
\p(t):= \sqrt{1-t^2}+ \frac{\pi-\arccos t}{t}>\pi \qquad\text{for any }t\in(0,1)
$$
which would allow to conclude. Since $\p(1)=\pi$, it is sufficient to show that $\p'<0$ in $(0,1)$. Indeed it is
\begin{eqnarray*}
\p'(t) &=& -\frac{t}{\sqrt{1-t^2}} + \frac{1}{t\sqrt{1-t^2}}+\frac{\arccos t-\pi}{t^2}\\
&\leq& -\frac{t}{\sqrt{1-t^2}} + \frac{1}{t\sqrt{1-t^2}}-\frac{\:\tfrac\pi2\:}{t^2}\\
&=& \frac{-t^3+t-\tfrac\pi2\sqrt{1-t^2}}{t^2\sqrt{1-t^2}}\\
&=& \frac{\sqrt{1-t^2}\left( t\sqrt{1-t^2}-\tfrac\pi2\right)}{t^2\sqrt{1-t^2}}\\
&<& \frac{\left( 1-\tfrac\pi2\right)}{t^2}\ <\ 0.
\end{eqnarray*}

{\em Step 3: For $v$ close to $\la\pi$ the isoperimetric set is of type $(b)$.}

\noindent
To prove this, it is sufficient to exhibit a set $F$ of area $v$ with ${P_f(F)}^2/v<4\pi$, which is the isoperimetric ratio of any ball of type $(a)$. Let $\ep:=\la\pi-v>0$, and remove from $B$ a little
curvilinear polygon of weighted area $\ep$ by a chord $PQ$ as in Figure \ref{ball}. Clearly, the remaining set $F_\ep$ is such that 
$P_f(F_\ep)\to 2\pi$ and $|F_\ep|_f\to\la\pi$ as $\ep\to0$, and so the 
isoperimetric ratio converges to $4\pi/\la<4\pi$. 
 
\begin{figure}[h!tbp]
\begin{center} 
\scalebox{1}{ 
\input{ball.pstex_t} 
} 
\end{center}
\caption{Area $\la\pi-\ep$.}
\label{ball} 
\end{figure}

{\em Step 4: There exists $0<v_0<\lambda\pi$ such that for $v<v_0$ the isoperimetric set is of type (a), and
for $v_0<v<\lambda\pi$ is of type (b).}

\noindent
Figure~\ref{graph_ballmagg} plots the perimeter as a function of the volume of type (a) configurations
(dashed line) and type (b) configurations (solid line); the dotted line is the perimeter of balls 
containing the interior of 
the unit ball. The slope of these graph is given by the curvature of the boundary of the isoperimetric set.
\begin{figure}[htbp]
\begin{center} 
\includegraphics[width=15cm]{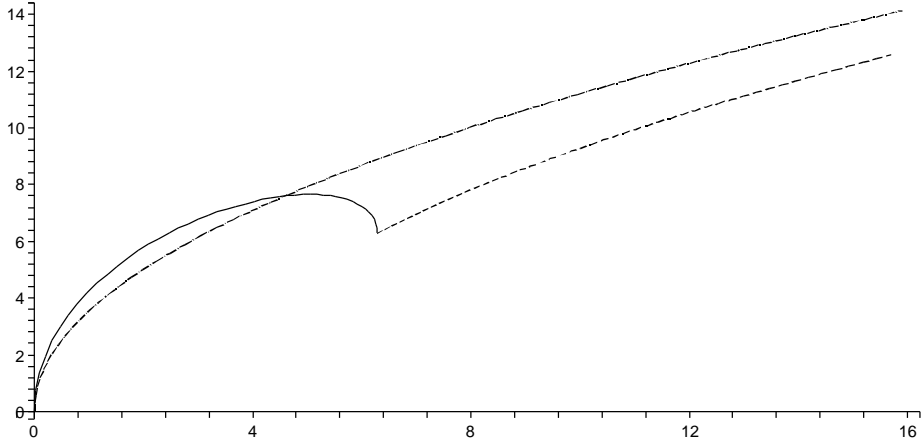}
\end{center}
\caption{Comparison between type (a) and type (b) configurations with $\lambda=2$.}
\label{graph_ballmagg} 
\end{figure}
We have to show that there exists a unique point $v_0>0$ such that the graph of the type (a) touches
the graph of type (b). We only consider convex type (b) sets, because
otherwise the curvature, and hence the slope of the (b)-graph, is negative. 
We denote by $B_b(v)$ the type (b) set with given volume $v$; it is given by
$B_b(v)=B\cap B_{r_b(v)}(x(v))$ for some $x(v)\in \R^2$ and with
$$
\frac{1}{r_b(v)}=\frac{d}{dv}P_f(B_b(v)).
$$
The Euclidean volume of $B_b(v)$ is a given percentage $\vrho(v)$ of the Euclidean volume of $B_{r_b(v)}(x(v))$, that is
$$
|B_b|_f=\lambda \vrho(v)\pi r_b(v)^2.
$$
We notice that $v\mapsto \vrho(v)$ is a decreasing function; to see this one can perform a linear transformation
in such a way that $B_{r_b(v)}(x(v))$ becomes the unit ball centered in the origin. With this modification,
the ball $B$ is sent to the ball $B'$ with $\partial B'$ meeting $\partial B_1(0)$ at a fixed angle. Enlarging $v$, we obtain
a smaller ball $B'$, so that $\varrho$ decreases.

We also denote by $B_a(v)=B_{r_a(v)}(y(v))$ a ball of type (a), with $r_a(v)=(v/\pi)^{1/2}$. If $v_0$ is the 
first value of the volume for which $P_f(B_b(v_0))=P_f(B_a(v_0))$, we have that $\pi r_a(v_0)^2=\lambda \vrho(v_0)\pi r_b(v_0)^2$ 
and, using the slopes of the graphs, 
$$
\frac{1}{r_b(v_0)}=\frac{d}{dv} P_f(B_b(v))_{|v=v_0}\leq
\frac{d}{dv} P_f(B_a(v))_{|v=v_0}=
\frac{1}{r_a(v_0)};
$$
this implies that $\vrho(v_0)\leq \frac{1}{\lambda}$. 
Then it follows that $\vrho(v)<\frac{1}{\lambda}$ for $v>v_0$, and so
$\pi r_a(v)^2=\lambda \vrho(v)\pi r_b(v)^2<\pi r_b(v)^2$, whence 
$$
\frac{1}{r_b(v)}< \frac{1}{r_a(v)}
$$
and then for $v>v_0$ the function $P_f(B_b(v))$ is strictly smaller than $P_f(B_a(v))$, due to the above inequality of slopes.

By numerical computation,
all of the cases (b)--convex, (b)--chord, and (b)--nonconvex occur for $\lambda$ close to 1, while for great $\lambda$ only (a) and (b)--nonconvex are isoperimetric.
\end{proof}

\begin{rmk}
\label{evolution}
Type (b) sets evolve from convex to nonconvex when they are isoperimetric solutions. 
Consider the curvature of the inner arc as a parameter; 
since all these inner arcs will meet $\partial B$ with constant angle, it turns that 
the enclosed volume is a decreasing function with respect to the curvature. 
So when volume increases, curvature of inner arcs goes from positive to negative, and hence 
type (b) sets go from convex to nonconvex.
\end{rmk}

\subsubsection{The ball with density $\la<1$.}
\label{ss:bmenor1}
We begin with describing the candidate isoperimetric sets; {\em mutatis mutandis}, most of the arguments are the same as in 
Section~\ref{ballmayor} for $\la>1$.

\begin{prop}
The candidate isoperimetric sets are (see Figure \ref{figbm1}):
\begin{itemize}
\item[(A)] balls entirely contained in $B$;
\item[(B)] sets enclosed by two arcs with common endpoints: the first arc lies on $\partial B$, the second one is in $\R^2\setminus B$ and meets $\partial B$ with angle $\arccos \la$;
\item[(C)] balls meeting $B$ orthogonally.
\end{itemize}

\begin{figure}[h!tbp]
\begin{center} 
\scalebox{1}{ 
\input{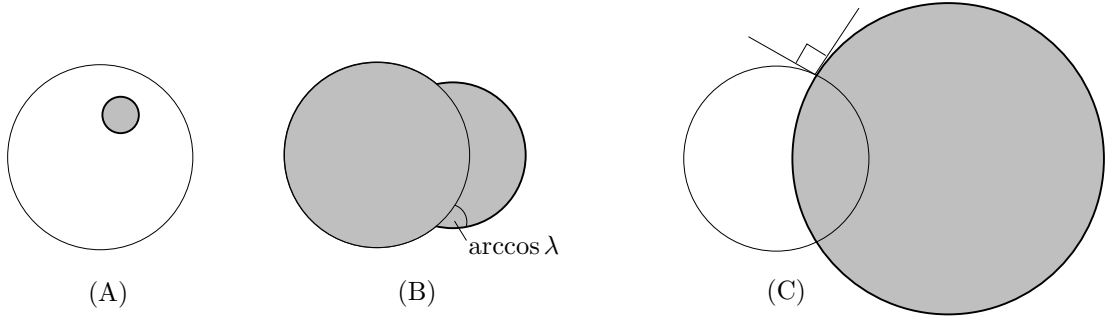} 
} 
\end{center}
\caption{Candidate isoperimetric sets for the ball density with $\la<1$.}
\label{figbm1} 
\end{figure}
\end{prop}

\begin{proof}
We closely follow the proof of Proposition~\ref{ballmino}.
In many cases, it will be sufficient to exchange the roles of internal and external arcs. We now agree to call {\em internal} an arc lying on $\partial B$. Again, we denote by $E$ an isoperimetric set of area $v$ and by $\g$ its boundary. 

{\em Step 1: $E$ is either a ball inside $B$ (i.e. a set of type (A)) or the closure of each of its components intersects both $\overline B$ and $\R^2\setminus \overline B$.}

This point presents no substantial difficulty with respect to Step 1 in Proposition \ref{ballmino}. 
In the rest of the proof we will suppose that $E$ is not a type-(A) ball, i.e. that the closure of any of its component intersects both $\overline B$ and $\R^2\setminus \overline B$; in particular, it cannot be a ball in $\R^2\setminus \overline B$. Notice also that the boundary of any component must meet $\partial B$: otherwise this component would be a ball enclosing $B$ and with boundary out of $\overline B$, a configuration which is easily proved to be worse than a ball outside $B$ with same area.

{\em Step 2: Only a finite number of external arcs are allowed.}

This is due to the fact that any external arc must leave $\partial B$ with an angle between $\arccos\la$ and $\pi-\arccos\la$; curvature being constant, this means that each arc must determine on $\partial B$ a corresponding arc whose length has a fixed positive lower bound. So only a finite number of arcs on $\partial B$ can be determined, and the claim easily follows.

{\em Step 3: Internal arcs cannot meet $\partial B$ with angle strictly greater than $\pi/2$.}

{\em Step 4: Only a finite number of internal arcs is allowed: therefore $E$ has regular traces on $\partial B$.}

{\em Step 5: The isoperimetric set is connected, enclosed by a piecewise $\ci^1$ Jordan curve, and the Snell law \eqref{eqfresnel} holds.}

These proofs are exactly the same as in the corresponding Steps of Proposition \ref{ballmino}: it is sufficient to exchange the words {\em internal} and {\em external}.

{\em Step 6: An internal arc cannot meet $\partial B$ transversally, unless the contact angle is $\pi/2$.}

To prove this, one can follow again the proof of the corresponding Step 6 in Proposition \ref{ballmino}, where however we excluded balls of type (C). This cannot be done in the $\la<1$ setting: we will instead show numerically that such a configuration can be isoperimetric (see the proof of Theorem \ref{conj:ballmino1}).

{\em Step 7: Reduction to cases (B) and (C).} 

We have shown that internal arcs can be only tangential or orthogonal to $\partial B$, and that there are only finitely many of them. If tangential, they have to lie on $\partial B$; otherwise they would have curvature more than 1, thus being complete circles inside $B$. Therefore, a connected component whose boundary possesses one (or more) arc on $\partial B$ must be bounded by exactly one external arc and one arc on $\partial B$; otherwise one could rotate the external arcs until forbidden meetings are obtained. Such a component falls into case (B) and must enclose the ball $B$; moreover, the angle between the external arc and $\partial B$ must be $\arccos \la$.

We have shown that the components can be only of type (B) or (C). Were there more than one, we could use rotation arguments to obtain a contradiction. This concludes the proof.
\end{proof}

Figure \ref{graph_ballmino2} shows the graphs of perimeter as function of the enclosed area for candidates in the case $\la=1/2$: the dotted, dashed and solid lines refer, respectively, to candidates (A), (B) and (C). 
Then it is then not surprising the following Theorem.

\begin{figure}[htbp]
\begin{center} 
\includegraphics[width=15cm]{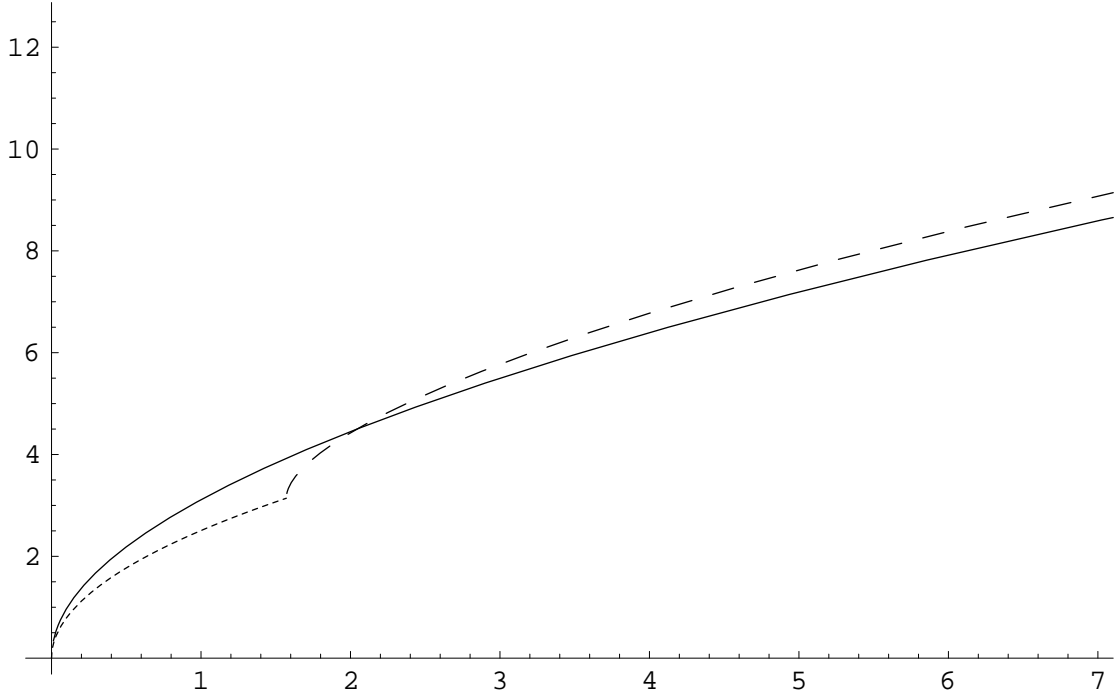}
\end{center}
\caption{Perimeter as function of area for three candidates in ball density with $\la<1$.}
\label{graph_ballmino2} 
\end{figure}

\begin{theo}\label{conj:ballmino1}
For the plane with density $\la<1$ on the unit ball and $1$ elsewhere, 
there exist some values $v_1, v_2>\la\pi$ such that the isoperimetric set of volume $v$ 
is given by
\begin{enumerate}
\item a ball of type (A) if $v\leq \la\pi$;
\item a set of type (B) if $\la\pi\leq v\leq v_1$;
\item a set of type (B) or (C) if $v_1<v< v_2$;
\item a ball of type (C) if $v\geq v_2$.
\end{enumerate}
\end{theo}

\noindent{\bf Remark.}
We believe that $v_1=v_2$ in the previous Theorem~\ref{conj:ballmino1}; 
that is, once sets of type (B) fail to be isoperimetric, type (C) sets are the solutions 
for larger volumes.  
We have observed this fact numerically, although we are not able to give a rigorous proof of it. 
\newline

\begin{proof}
Since it will be useful in the sequel, we compute explicitly the perimeter and area of candidates (B) and (C). For type-(B) sets, let $2\be$ be the angle determined by the external arc; by $\al=\be+\arccos \la$ we denote half of the angle at the center of $\partial B$ formed by the arc on $\partial B$ (see Figure~\ref{fig:alphabeta}). It is not difficult to show that $0\leq\be\leq\pi-\arccos\la$ and that, after setting $r:=\sin\al/\sin\be$ to be the radius of the external arc, the perimeter is
$$
P_B=2\big( (\pi-\be)r+\la\al\big)
$$ 
while the enclosed area is
$$
A_B=r^2(\pi-\be+\sin\be\cos\be)+(\al-\sin\al\cos\al)-(1-\la)\pi\,.
$$

\begin{figure}[ht]
\centerline{\includegraphics[width=0.25\textwidth]{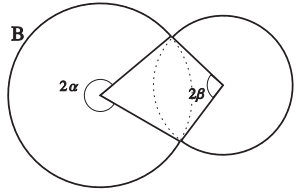}}
\caption{Type-(B) set}
\label{fig:alphabeta}
\end{figure}

Regarding balls of type (C), if $2\hat\be$ is the angle at the center determined by the internal arc, then the radius of the ball will be $\hat r:=1/\tan\hat\be$. Therefore, the perimeter will be given by
$$
P_C=2\frac{\pi-(1-\la)\hat\be}{\tan\hat\be}
$$
and the area enclosed by
$$
A_C:=\frac{\pi-(1-\la)(\hat\be-\sin\hat\be\cos\hat\be)}{\tan^2\hat\be}-(1-\la)\left(\tfrac\pi2-\hat\be-\sin\hat\be\cos\hat\be \right)\,.
$$
Since type-(B) sets enclose the ball $B$, for area bounds less than $\la\pi$ we have only to compare candidates (A) and (C). To prove the first part of the statement it will be sufficient to show that the ratio $P_C^2/A_C$ is strictly greater than $4\la\pi$, which is the same ratio computed for (A)-balls. Let us rewrite $A_C$ as
\begin{eqnarray*}
A_C &=& \frac{\pi-(1-\la)(\hat\be-\sin\hat\be\cos\hat\be)}{\tan^2\hat\be}-(1-\la)\left(\tfrac\pi2-\hat\be-\sin\hat\be\cos\hat\be \right)\\
&=& \tfrac1{\tan^2\hat\be}\left( \pi-(1-\la)\hat\be-(1-\la)\tan\hat\be(-\cos^2\hat\be+(\tfrac{\pi}{2}-\hat\be)\tan\hat\be-\sin^2\hat\be)\right)\\
&=& \tfrac1{\tan^2\hat\be}\left( \pi-(1-\la)\hat\be-(1-\la)\tan\hat\be((\tfrac{\pi}{2}-\hat\be)\tan\hat\be-1)\right)\,.
\end{eqnarray*}
The desired inequality $P_C^2>4\la\pi A_C$ is equivalent to
$$
\left( \pi-(1-\la)\hat\be\right)^2\:>\:\la\pi \left( \pi-(1-\la)\hat\be-(1-\la)\tan\hat\be\big((\tfrac{\pi}{2}-\hat\be)\tan\hat\be-1\big)\right)
$$
i.e.
$$
(1-\la)\pi^2\:>\:(1-\la)\left[ 2\pi\hat\be-(1-\la)\hat\be^2-\la\pi\hat\be-\la\pi\tan\hat\be\big((\tfrac{\pi}{2}-\hat\be)\tan\hat\be-1\big)\right]
$$
and in turn to
$$
-2\pi\big(\tfrac{\pi}{2}-\hat\be\big)-(1-\la)\hat\be^2-\la\pi\hat\be-\la\pi\big(\tfrac{\pi}{2}-\hat\be\big)\tan^2\hat\be+\la\pi\tan\hat\be<0\,.
$$
Remembering that $\hat\be\in[0,\pi/2]$, setting $\theta:=\tfrac{\pi}{2}-\hat\be$ and discarding the negative term $-(1-\la)\hat\be^2-\la\pi\hat\be$, it will be sufficient to show that
$$
2\pi\theta+\frac{\la\pi\theta}{\tan^2\theta}-\frac{\la\pi}{\tan\theta}=\frac{\pi}{\tan^2\theta}\left[2\theta\tan^2\theta+\la\theta-\la\tan\theta \right]>0\,.
$$
Recalling that $\theta\in[0,\pi/2]$ it will suffice to prove that $\p(\theta):=2\theta\tan^2\theta+\la\theta-\la\tan\theta$ is positive in $[0,\pi/2]$. Indeed it holds $\p(0)=0$ and
$$
\p'(\theta)=2\tan^2\theta+4\theta\tan\theta(1+\tan^2\theta)+\la-\la-\la\tan^2\theta>0
$$
which allows to conclude because $\la<1$.

For area bounds greater than $\la\pi$ we have only to compare candidates (B) and (C). 
From above we have that the type-(C) ball for area $\la\pi$ 
has perimeter greater than $2\la\pi+\de$, with $\de>0$; 
this will happen for any (C)-ball with area greater than $\la\pi$. The second part of the statement follows by noticing that, when the area $A_B$ of the (B)-candidate decrease to $\la\pi$ (equivalently when $r\downarrow 0$, i.e. $\al\uparrow\pi$, i.e. $\be\uparrow\pi-\arccos\la$), its perimeter $P_B$ decreases to $2\la\pi$.

When considering a type (B)-candidate associated with an external ball
of radius $r$, straightforward computations give
$$
A_B(r)=\pi r^2+O(1)\quad\text{and}\quad P_B(r)=2\pi
r+2\arccos\la-2\sqrt{1-\la^2}+O(1/r^2)
$$
for large $r$. Analogously, for a type-(C) ball of big radius $r_C$ we have
$$
A_C(r_C)=\pi r_C^2+O(1)\quad\text{and}\quad P_C(r_C)=2\pi
r_C-2(1-\la)+O(1/r_C^2)\,.
$$
By imposing the equality $A_C(r_C)=A_B(r)$ of enclosed areas, in order
to compute $r_C$ as function of $r$, one obtains $r_C=r+O(1/r)$. To
show that $P_C(r_C(r))<P_B(r)$ for large areas, we just need to prove that
$$
-2(1-\la)+O(1/r) < 2\arccos\la-2\sqrt{1-\la^2}+O(1/r^2)
$$
for large $r$. Therefore it will be sufficient to show that
$f(\la):=\arccos\la-\sqrt{1-\la^2}+1-\la$ is strictly positive for any
$\la\in]0,1[$: this in turn follows since $f(1)=0$ and
$$
f'(\la)=-\frac{1-\la}{\sqrt{1-\la^2}}-1=-\sqrt{\frac{1-\la}{1+\la}}-1
<0\qquad\text{for any }\la\in]0,1[\,.
$$ 
\end{proof}

\section{Modifications of Gauss density on the plane}
\label{SecContDens}


This section provides partial results on the isoperimetric problem for certain modifications of Gaussian density $Ce^{-cr^2}$ on the plane. Previously, Borell \cite{Bor86} and Rosales et al. \cite[Introduction and Section 5]{BayCanMorRos07OnT} proved by sy\-mme\-trization that for the density $e^{r^2}$ isoperimetric regions are balls about the origin.

Recently, Brock et al. \cite{BroChiMer08ACl} proved that for the modified Gauss space $x_N^ke^{-|x|^2}$, $k\geq 0$, 
on the half--space $\R^N_+$, isoperimetric sets are bounded by vertical hyperplanes. 
The same argument also shows that the same holds if we consider the density $|x_N^k|e^{-|x|^2}$
on the whole space $\R^N$, simply by considering the reflection with respect to the 
hyperplane $x_N=0$. In fact, if $E$ is an isoperimetric set, then $E\cap \R^N_+$ is isoperimetric
with respect to its volume bound, so its boundary in $\R^N_+$ has to be a vertical hyperplane. 

It is worth noticing that in the planar case $N=2$ and with $k=1$, vertical and horizontal lines
are the only straight lines with constant generalised curvature; moreover, the only circles with constant
generalised curvature are the one centered at the origin.

Motivated by these results, we will now consider the density in the plane defined by
\begin{equation}
\label{GaussDens2}
f(x,y)=\exp(-x^2-y^4),\quad (x,y)\in\R^2.
\end{equation}

Since the total mass is finite we have existence of isoperimetric regions for any prescribed volume.
We state the following 

\begin{conjecture}
\label{conj:1}
For density (\ref{GaussDens2}) on the plane, every isoperimetric region is bounded by a horizontal line (for prescribed volumes close to half the total volume) or by a vertical line.
\end{conjecture}

We trivially have the following lemma. 

\begin{lemma}
Vertical and horizontal lines are the unique straight lines with
cons\-tant generalised mean curvature.
Furthemore, circles do not have constant generalised mean curvature.
\end{lemma}

A first consequence of the previous lemma is that no arc of a circle will be
contained in the boundary of an isoperimetric region.
We now proceed to compare the perimeter of vertical and horizontal
lines enclosing the same volume.
We shall see that, for some volumes,
vertical lines are better, while horizontal lines beat vertical ones for other volumes.

\begin{prop}\label{propconj1}
For volume fractions near $1/2$, horizontal lines have less perimeter than vertical lines, 
and for small or large volumes, vertical lines have less perimeter than horizontal ones.  
\end{prop}

\begin{proof}
For $y\in[0,+\infty)$,
let $R_h(y)$ be the horizontal line $\{(x,y):x\in\R\}$.
This line bounds a region with volume equal to
\[
v_h(y)=\frac{\sqrt{\pi}}{4}\,\Gamma\big(\tfrac{1}{4},y^4\big),
\]
where $\Gamma$ represents the Gamma function of two arguments. 
On the other hand, for the perimeter of the region bounded by $R_h(y)$ we have
\[
P_h(y)=\sqrt{\pi}\,e^{-y^4}.
\]

For $x\in[0,+\infty)$,
let $R_v(x)$ be the vertical line $\{(x,y):y\in\R\}$. Then,
the region bounded by this line has a volume equal to
\[
v_v(x)=\sqrt{\pi}\ \Gamma\big(\tfrac{5}{4}\big)(1-{\rm Erf}(x)), 
\]
where $\rm Erf$ is the error function. 
Moreover, the perimeter of the region bounded by $R_v(x)$ is given by
\[
P_v(x)=2\,\Gamma\big(\tfrac{5}{4}\big)\,e^{-x^2}.
\]

Given $y$, $x\in[0,+\infty)$, it is easy to check that 
$v_h(y)=v_v(x)$ if and only if
\begin{equation}
\label{eq:samevolume}
{\rm Erf}(x)=1-\frac{\Gamma(\frac{1}{4},y^4)}{4\,\Gamma(\frac{5}{4})}.
\end{equation}

Consider now $y$, $x\in[0,+\infty)$
satisfying condition~\eqref{eq:samevolume}.
Then
\begin{align}
\label{eq:condition}
P_v(x)<P_h(y)&\Leftrightarrow
\frac{e^{-x^2}}{e^{-y^4}}<\frac{\sqrt{\pi}}{2\Gamma(\frac{5}{4})}\notag
\\
&\Leftrightarrow
\log\bigg(\frac{e^{-x^2}}{e^{-y^4}}\bigg)<\log\bigg(\frac{\sqrt{\pi}}{2\Gamma(\frac{5}{4})}\bigg)\notag
\\
&\Leftrightarrow
-x^2 + y^4 < \log\bigg(\frac{\sqrt{\pi}}{2\Gamma(\frac{5}{4})}\bigg)\sim -0.02251.
\end{align}
Hence, if this last condition is satisfied,
the corresponding vertical line is isoperime\-trically better than the
horizontal line.

Observe that the left hand side in condition \eqref{eq:condition} can be expressed as a function of
$y$, by using \eqref{eq:samevolume}.
Numerical computations show that for
values of $y$ close to zero,
e.g. when
\[
y\in[-0.15,0.15],
\]
the left hand side in~\eqref{eq:condition} is greater than $-0.02251$,
and so, horizontal lines are better.
For $$|y|>0.16,$$ such a term is less than $-0.02251$
and so vertical lines are better. This completes the proof. 
\end{proof}

\begin{rmk}
In this situation, domains bounded by two (symmetric) lines can be easily discarded,
since the perimeters of horizontal and vertical lines are strictly decreasing functions.
\end{rmk}

\begin{rmk}
Proposition \ref{propconj1} indicates that vertical or horizontal lines may appear as
minimisers.
As in \cite{BroChiMer08ACl}, 
given a set $\Omega\subset\R^2$, we could apply two adapted Steiner symmetrisations
(one in the vertical direction and another one in the horizontal direction), 
to obtain another set $\Omega^*\subset\R^2$ which would have 
the same volume and no more perimeter than $\Omega$.
\end{rmk}

\begin{rmk}
In this setting, it can be checked that horizontal lines far away from the origin 
(i.e. for small or large volumes) are unstable by using the second variation formula 
(see \cite[Proposition~3.6]{BayCanMorRos07OnT}) 
with a suitable mean zero function. 
Unfortunately, the study of the stability of the rest of the lines seems a harder question. 
\end{rmk}

\section{Open questions}\label{SecOpenQue}

The study of manifolds with density, and of the isoperimetric problem in parti\-cu\-lar, is relatively recent, and many 
related works have appeared during the last years \cite{Mor05Man,CJQW,BayCanMorRos07OnT}.
In this section, we will list several open problems in this setting: some of them (Questions 1--4) 
are closely related to the present paper, while the other ones arose during Morgan's lectures on
``Manifolds with density'' at the GMTLAP meeting in Modena, February 2007.

\begin{que}[The strip case]\label{questrip}
Conjecture \ref{conj:strip} remains open: as we said we think that a four--arc candidate (iv) cannot be isoperimetric, but we are presently not able to prove this.
%
\end{que}

\begin{que}[The ball case]\label{queball}
The case $\lambda<1$ in Section~\ref{ss:bmenor1} is not completely solved (see Theorem~\ref{conj:ballmino1}): 
it is left to prove that, once sets of type (C) are isoperimetric, type (B) sets cannnot be solutions for any other 
larger volume.  

It would be also interesting to investigate the ball density problem in
$\R^N$, $N\geq 3$. For $\la>1$ and volume
$v\geq \la\omega_N$ balls containing $B$ are
isoperimetric: the proof is the same as the planar case.

One could further imagine densities which are constantly 1 on $\R^2$ except for some balls:
this situation seems particularly interesting, since it may well provide non--connected
isoperimetric sets. 
Consider for example the piecewise constant density 
$\la$ on $B\cup B'$ and 1 outside, where $\la\gg 1$ and $B$, $B'$ are two sufficiently distant balls of radius one.
It is not difficult to show that isoperimetric sets, at least for some
area bounds, are not connected. In fact, take $2\pi\la$ as prescribed area:
were the isoperimetric set connected, by Theorem \ref{ballmagg} it
would consist of a ball containing $B$ or $B'$, but not both of them 
as they are chosen far away (recall the diameter estimate \eqref{estdiamper}).
Therefore, its density perimeter would be at least $2\pi\sqrt{\la}$, while $B\cup B'$ 
satisfies the area bound with perimeter $4\pi$, which is smaller for large values of $\la$.
\end{que}


\begin{que}[Plane with density $\exp(-x^2-y^4)$]\label{quex2y4}
Conjecture \ref{conj:1} (iso\-peri\-me\-tric boundaries are given by horizontal or vertical lines) is 
still open, since we have not determined other possible competitors.
\end{que}

\begin{que}\label{quebobkov}
Other modifications of the Gaussian density on the plane would be worth investigating. For instance, we propose the 
density $f(x,y):=\exp(-x^4-y^4)$ on $\R^2$. An interesting approach seems to be related to a work by S.Bobkov \cite{Bob96Ext}, 
where it is proved that half-spaces are isoperimetric in $\R^N$ with density given as a product $f(x)=\varphi(x_1)\cdots \varphi(x_N)$, 
under certain assumptions on $\varphi$. Unfortunately, these hypothesis are not quite 
satisfied in our case $\varphi(x)=e^{-x^4}$. 
\end{que}

\begin{que}[Proposed by A. Brancolini]
Consider $\R^2$ endowed with a density attaining its minimum value
in several points.
Then it is reasonable to think that for small prescribed areas
an isoperimetric region should be a nearly round disc containing one of these points. The reader may compare this with \cite[\S~5]{BayCanMorRos07OnT}, where this result is proven for a particular density with a unique minimum point.
\end{que}

\begin{que}[Proposed by A. Brancolini]
Another interesting question consists of finding the iso\-pe\-ri\-me\-tric regions in Euclidean space
endowed with a density which is a finite sum of
Gaussian measures with different barycenters and
variances.
In the particular case that all the barycenters lie in the same line $\ell$,
orthogonal lines to $\ell$ seem nice candidates for being the solutions.
\end{que}

\begin{que}
Which regions (with rectifiable boundary) are isoperimetric for some density? Probably any smooth region has this property.
Take
a density with very low values along the boundary, perhaps a low constant value with normal derivative determined by the constant generalised mean curvature condition.
\end{que}

\begin{que}[Proposed by S. Ansaloni]
Consider the sphere $S^2$ endowed with some density such as
$$
f(\theta,\phi):=(1-\sin \phi)^{-p}, 
$$
where $\phi\in[-\tfrac{\pi}{2},\tfrac{\pi}{2}]$ is the latitude. Here the conjectured isoperimetric sets are polar caps.
\end{que}

\begin{que}[Proposed by A. Pratelli] 
Fairly recent ``quantitative'' isoperimetric inequalities show that a
set with perimeter close to the optimal one must approximate a minimiser in
shape (see \cite{FusMagPra05}). Can such results be generalised to manifolds with
density?
\end{que}

\begin{que}
Suppose the density $f$ is ``conic'', i.e. constant along rays starting from the origin. 
Give necessary conditions for a cone $C$ to be the boundary of a perimeter--minimising set. The main case is given by the half-space density and the Snell refraction law is one necessary condition. Such results could be used 
more generally. Take a point of an isoperimetric boundary which is also a discontinuity point for the density, and perform a blow-up. 
In the limit the above situation is retrieved.
\end{que}

\end{document}